\documentclass[notitlepage,11pt, reqno]{amsart}
\usepackage{color}
\usepackage{amssymb,nicefrac,cite,bm,upgreek}
\usepackage[mathscr]{eucal}
\usepackage{graphicx}
\usepackage{epsf,times}

\definecolor{dmagenta}{rgb}{.4,.1,.5}
\definecolor{007}{rgb}{.0,.0,.7}
\definecolor{dred}{rgb}{.5,.0,.0}
\definecolor{dgreen}{rgb}{.0,.5,.0}
\definecolor{dblue}{rgb}{.0,.0,.5}

\definecolor{violet}{rgb}{.3,.0,.9}
\definecolor{orange}{cmyk}{0,.5,.1,.0}
\definecolor{dcyan}{cmyk}{.5,.0,.0,.0}
\definecolor{dyellow}{cmyk}{.0,.0,.5,.0}
\definecolor{cm}{cmyk}{1,.0,.0,.0}
\numberwithin{equation}{section}
\allowdisplaybreaks

\newtheorem{theorem}{Theorem}[section]
\newtheorem{lemma}{Lemma}[section]
\newtheorem{proposition}{Proposition}[section]

\theoremstyle{definition}
\newtheorem{definition}{Definition}[section]

\theoremstyle{remark}
\newtheorem{remark}{Remark}[section]

\newcommand{\cA}{\mathcal{A}}
\newcommand{\calK}{\mathcal{K}}   
\newcommand{\cP}{\mathcal{P}}

\newcommand{\grad}{\nabla}

\newcommand{\tc}{{\Breve\uptau}}
\newcommand{\Sob}{\mathscr{W}}
\newcommand{\Sobl}{\mathscr{W}_{\mathrm{loc}}}

\newcommand{\Lyap}{\mathcal{V}}
\newcommand{\Usm}{\mathfrak{U}_{\mathrm{SM}}}

\DeclareMathOperator{\Exp}{\mathbb{E}}
\DeclareMathOperator{\Prob}{\mathbb{P}}
\newcommand{\D}{\mathrm{d}}

\newcommand{\M}{\mathbf{M}}
\newcommand{\R}{\mathbb{R}}
\newcommand{\B}{\mathbb{B}}
\newcommand{\Rd}{\mathbb{R}^d}
\newcommand{\NN}{\mathbb{N}}

\newcommand{\Uadm}{\mathfrak{U}}
\newcommand{\I}{\mathbf{I}}

\newcommand{\calB}{\mathcal{B}}
\newcommand{\calC}{\mathcal{C}}

\newcommand{\calF}{\mathcal{F}}

\newcommand{\sF}{\mathfrak{F}}

\newcommand{\transp}{^{\mathsf{T}}}

\newcommand{\al}{\alpha}

\newcommand{\lam}{\lambda}
\newcommand{\gam}{\gamma}
\newcommand{\kap}{\kappa}
\newcommand{\lfl}{\lfloor}
\newcommand{\rfl}{\rfloor}
\newcommand{\eps}{\varepsilon}

\newcommand{\del}{\delta}

\newcommand{\bvnorm}[1]{[\kern-0.45ex[\kern0.1ex #1 \kern0.1ex]\kern-0.45ex]}

\newcommand{\abs}[1]{\lvert#1\rvert}
\newcommand{\norm}[1]{\lVert#1\rVert}

\newcommand{\order}{{\mathscr{O}}} 
\newcommand{\sorder}{{\mathfrak{o}}} 

\newcommand{\df}{:=}
\DeclareMathOperator*{\Argmin}{arg\,min}

\DeclareMathOperator{\diag}{diag}

\begin{document}

\title[Control of queueing systems with help]
{An ergodic control problem for many-server multiclass queueing systems with cross-trained servers}

\author{Anup Biswas}
\address{Department of Mathematics,
Indian Institute of Science Education and Research, 
Dr. Homi Bhabha Road, Pune 411008, India.}
\email{anup@iiserpune.ac.in}

\date{\today}

\begin{abstract}
A $M/M/N+M$ queueing network is considered with $d$ independent customer classes and $d$ server pools in Halfin-Whitt regime.
Class $i$ customers has priority for service in pool $i$  for $i=1, \ldots, d$,
and may access some other pool if the pool  has an idle server and all the servers in pool $i$ are busy. We formulate an ergodic control
problem where the running cost is given by a non-negative convex function with polynomial growth. We show that
the limiting controlled diffusion is modeled by an action space which depends on the state variable. 
We provide a complete analysis
for the limiting ergodic control problem and establish asymptotic convergence of the value functions for the queueing  model.
\end{abstract}



\maketitle

\section{Introduction}\label{S-intro}

In this article we consider a queueing system consisting of
 $d$ independent customer classes and $d$ server pools (or stations).
Each server pool  contains $n$ identical servers. Customers of class $i$ have priority for service in  pool $i$
and this priority is of preemptive type. Therefore a newly arrived job of class $i$ at time $t$ would preempt the service of
a class $j$ job, $j\neq i$, if there is a class $j$ job receiving service from pool $i$ at time $t$.
A customer
from class $j$ may access service from pool $i, j\neq i,$ if and only if there is an empty server in the pool $i$ and all
the servers in the  pool $j$ are busy. 
Therefore service stations are  \textit{cross-trained} to serve nonpriority customers when its own priority customer class is {\it underloaded}.
Customers are also allowed
to renege the system from the queue. The arrival of customers are given by $d$-independent Poisson processes with
parameter $\lam^n_i, i=1, \ldots, d$. By $\mu^n_{ij}$ we denote the service rate of class $i$ customers at station $j$. 
We shall use $\mu_i$ instead of $\mu_{ii}$ for $i=1, \ldots, d$.
The network is assumed to work under Halfin-Whitt setting in the sense that
\begin{equation}\label{C}
\lim_{n\to\infty}\sqrt{n}(1-\frac{\lam^n_i}{n\mu_{i}})\quad \text{exists, for all}\; \; i\in\{1, \ldots, d\}.
\end{equation}
Therefore under \eqref{C} each customer class $i$ is in criticality with respect to pool $i$, $i\in\{1, \ldots, d\}$, i.e., $n\approx\mathbf{r}^n_i + \rho_i \sqrt{\mathbf{r}^n_i}$  for 
some constant $\rho_i$ where
$\mathbf{r}^n_i=\nicefrac{\lambda^n_i}{\mu_{i}}$ is the mean offered load of class $i$ to the pool $i$.
Note that the above criticality condition is different from those generally used in multiclass multi-pool setting \cite{atar-2005}.
This condition could be seen as a generalization to \cite[Assumption~4.12(3)]{ramanan-reiman-1} to the many-server setting.
To elaborate we recall the generalized processor sharing (GPS) network from \cite{ramanan-reiman-1}. In a multiclass GPS network with
$d$ customer classes and single server
the server would use a fraction $\alpha_i$, $(\alpha_1, \ldots, \alpha_d)\in (0, \infty)^d$ is a given probability vector, of the total processing capacity to serve class-$i$ jobs when all the job classes are available in the system, otherwise (that is when positive number of classes are empty) any excess processing 
capacity, normally reserved for the job classes that are empty,
is redistributed among the nonempty classes in proportion to the weight vector $(\alpha_1, \ldots, \alpha_d)$. In this case the conventional heavy traffic condition implies that
$\lim_{n\to\infty}\sqrt{n}(\nicefrac{\mathbf{r}^n_i}{n}-\alpha_i)$ exists for all $i$, see \cite[Assumption~4.12(3)]{ramanan-reiman-1}. Therefore \eqref{C} can be seen as a many-server
analogue to the conventional heavy-traffic condition of GPS network.

Control is given by a matrix value process $Z$ where $Z_{ij}$ denotes the number of class-$i$ customers getting served at 
station $j$. We note that for our model a control $Z$ need not be work-conserving.
The running cost is given by $r(\Hat{Q}^n)$ where $r$ is a non-negative, 
convex function of polynomial growth and $\Hat{Q}^n$
denotes the diffusion-scaled queue length vector i.e., $\Hat{Q}^n=\frac{1}{\sqrt{n}}Q$ where $Q$ is the queue length vector.
 We are interested in the cost function
$$\limsup_{T\to\infty}\;\frac{1}{T}\Exp\Bigl[\int_0^T r(\Hat{Q}^n(s))\, \D{s}\Bigr].$$
The value function $\widehat{V}^n$ is defined to be the infimum of the above cost where the infimum is taken over all admissible controls. One of
the main results of this paper analyze asymptotic behavior of $\widehat{V}^n$ as $n\to\infty$. In Theorem~\ref{T-optimality} we show that
$\widehat{V}^n$ tends to the optimal value of an ergodic control problem governed by certain class of controlled diffusions.
We also study the limiting ergodic control problem and establish the existence-uniqueness results of the value function
in Theorem~\ref{T-HJB}.
 It is worth mentioning
that results like Theorem~\ref{T-optimality} and ~\ref{T-HJB} continue to hold if
one considers other types of convex running cost functions which might depend on
 $\Hat{Z}^n$ (see Remark~\ref{con-rem-1}).
Let us denote by $i\rightarrow j$ ($i\nrightarrow j$) when class-$i$ customers can (resp., can not) access service
from station $j$. In this article
we have concentrated on the situation where $i\rightarrow j$, for all $i, j$, but it is not a necessary condition for our method to work. As noted in Remark~\ref{con-rem-2}, if we impose $i\nrightarrow j$ for some $i\neq j$ in the above queueing model, our results continue to hold without any major change in the arguments.

{\bf Literature review:}\, Scheduling control problems have a rich and profound history in queueing literature. 
The main goal of these
problems is to schedule the customers/jobs in a network in an optimal way. But it is not always possible
to find a simple policy that is optimal. It is well known that for various queueing networks with finitely many servers $c\mu$ policy is optimal
\cite{atar-biswas, biswas-2012, mandel-stolyar, mieghem}. $c\mu$  scheduling rule prioritize 
the service of the job classes according to the index $c_i\mu_i$ (higher index gets priority for receiving service) where $c_i$ denotes the holding cost for class-$i$ and $\nicefrac{1}{\mu_i}$ denotes the mean service time of class-$i$ jobs.
In case of many servers, it is shown in \cite{atar-giat-shim, atar-giat-shim-2} that a similar priority 
policy, known as $c\mu/\theta$-policy, that prioritize the jobs according to the index $c_i\mu_i/\theta_i$, $\theta_i$
being the abandonment rate of class-$i$,  is optimal when
the queueing system asymptotic is considered under fluid scaling and the cost is given by the long run average of a 
linear running cost. But existence of such simple optimal priority policies are rare in Halfin-Whitt setting. 
In general, by Halfin-Whitt setting we mean the number of servers $n$ and the total 
offered load $\mathbf{r}$ scale like $n \approx \mathbf{r} + \rho\sqrt{\mathbf{r}}$ for some constant $\rho$.
See \eqref{HWpara} below for exact formulation for our model.
 There are several papers devoted to the study of control problems in Halfin-Whitt regime. 
 \cite{atar-mandel-rei, harrison-zeevi} studied a control problem with discounted pay-off
 for multiclass single pool queueing network and  asymptotics of the value functions are obtained in \cite{atar-mandel-rei}.
 Later these works are generalized to multi-pool case in \cite{atar-2005}. \cite{atar-man-shai} considered
a simplified multiclass multi-pool control problem with a discounted cost criterion 
where the service rates either depend on the class or
the pool but not the both, and established asymptotic optimality for the scheduling policies.
 Under some assumptions on the holding cost,  a static priority policy is shown to be optimal in \cite{tezcan-dai, dai-tezcan-2011} 
in a multiclass multi-pool queueing network where the cost function is of finite horizon type. \cite{Gurvich-Whitt-2009} studied queue-and-idleness-ratio controls,
 and their associated properties and staffing implications for multiclass multi-pool networks.
 In \cite{ari-bis-pang}
the authors considered an ergodic control problem for multiclass many-server queueing network and established  convergence of the value functions. 
Some other works that have considered ergodic control problems for queueing networks are as follows: \cite{budhi-ghosh-lee} considers
an ergodic control problem in the conventional heavy-traffic regime and establishes asymptotic optimality, \cite{kim-ward} studies admission
control problem with an ergodic cost criterion for a single class $M/M/N+M$ queueing network. For an inverted 'V' model it is shown in \cite{Armony-05}
that the fastest-server-first policy is asymptotically optimal for minimizing the steady-state expected queue length and waiting time.
\cite{Ward-Armony-13} considered a blind fair routing policy for multiclass multi-pool networks and used simulations to validate the performance of the blind fair routing policies comparing them with non-blind policies derived from the limiting diffusion control problem. Recently, ergodic
control of multiclass multi-pool queueing networks is considered in \cite{ari-pang-16} where the authors have addressed existence and uniqueness
of solutions of the HJB (Hamilton-Jacobi-Bellman) equation associated to the limiting diffusion control problem. Asymptotic optimality results for
the N-netwrok queueing model are obtained in \cite{ari-pang-N}. Most of these above works \cite{atar-mandel-rei, harrison-zeevi,
atar-2005, atar-man-shai, ari-bis-pang}
on many server networks consider work-conserving policies as their \textit{admissible} controls. It is necessary to point out some key
 differences of the present queueing network with the one considered in \cite{atar-2005, ari-pang-16}. First of all the
 Halfin-Whitt condition above (see \eqref{C}) is different from \cite[p. 2614]{atar-2005} and therefore, the diffusion
 scalings of the customer count processes are also different. Moreover, the collection of admissible controls in \cite{atar-2005}
  includes a wider class of 
 scheduling policies which are jointly work-conserving and need not follow any class-to-pool priority, whereas
 for the queueing model under consideration every admissible control must obey the class-to-pool priority constrain.
 The particular nature of our network allow us to consider an optimal ergodic control problem with a general
 running cost function and to obtain asymptotic optimality (Theorem~\ref{T-optimality}) under standard assumptions on the
 service rates.

{\bf Motivation and comparison:}\, The above model is realistic in many scenario. For instance, in a 
call center different service stations are designed to serve certain
type of customers and they may choose to help other type of customers when one or many servers are idle in the station.
It is known that cross-training of customer sales representatives in inbound call center reduces the average number of
customers in queue. We refer to \cite{IKO} and the references therein for a discussion on labor cross-training and its effect on the performance on queueing networks. 
Since it is expensive to train every sales representative in all skills it becomes important to understand the optimal cross-training structure of the agents which
reduces the  average number of customers in queue. \cite{IKO} uses \textit{average shortest path length} as a metric to predict the more effective cross-training structures in terms
of customer waiting times. Our model is a variant of these networks.
As mentioned before, we may have $i\nrightarrow j$ for some $i\neq j$ in our queueing network which should be interpreted as the inability of station $j$ to serve
class $i$ jobs. It is also reasonable to have class-to-pool priority when the agents in pool $i$ are primarily trained to serve jobs of class $i, i\in\{1, \ldots, d\},$
 and might not be efficient in serving class $j, i\neq j$.
It should be noted that for our queueing model we have fixed a cross-training structure and we are trying to investigate the optimal dynamic scheduling policy that will optimize the pay-off function.

Our model also bears resemblance with Generalized Processor Sharing (GPS) models
from \cite{ramanan-reiman-1, ramanan-reiman-2}.
 In fact, our model can be scaled 
 to a single pool case where each class of customers have priority in accessing a fixed fraction of the total number of servers and they may get access to other
 servers, fixed for other customer classes, if the queues of other customer classes are empty. It should be observed that the multi-pool version is more general than the
 single pool version. For instance, it is not natural to have $\mu_{ij}\neq\mu_{i}$, for $i\neq j$, in the setting of a single pool with homogeneous
 servers, but this is not the case for a multi-pool model. Therefore we stick to the multi-pool network model.
A legitimate question for these processor sharing type model is that whether the GPS type policy is optimal or not for the pay-off function considered above. 
Motivated by this question a similar control problem is considered in \cite{BISW} for a queueing model with finitely many servers and it is shown that
the value function associated to the limiting controlled diffusion model solves a non-linear Neumann boundary value problem. The
solution is obtained in the viscosity sense and therefore, it is hard to extract any information about the optimal control, even for the 
diffusion control problem.
The present
work is also motivated by a similar question but for the many-server queueing network. 
 One  motivation of the
present work is to get some insight about the optimal control. The motivating question here is 
if we allow the non-priority classes of pool $i$ to access the servers of pool $i$ in some fixed proportion when pool $i$ has some free servers then such scheduling
policy would be optimal or not.
In the present work
we characterize the value functions of the queueing model by its limit and construct a
sequence of admissible policies that are asymptotically optimal.
Though theoretically we can find a minimizer for the limiting HJB, it is not easy to compute it explicitly or numerically. 
One of the future directions is to
compute the minimizer and compare its performance with the GPS type scheduling.

The  methodology of this problem is not immediate from any existing work. In general, the main idea is to convert
such problems to a controlled diffusion problem and analyze the corresponding Hamiltion-Jacobi-Bellman(HJB) equation to
extract information about the minimizing policies. All the exiting works \cite{atar-mandel-rei, harrison-zeevi, atar-2005, ari-bis-pang}
use work-conservative properties of the controls to come up with an action space that does not depend on the state variable.
But as we mentioned above that our policies need not be work-conserving. Also there is an obvious action space that
one could associate to our model (see \eqref{con-set}). Unfortunately, this action space depends on the state variable. In general,
such action spaces are not very favorable for mathematical analysis. Existence of measurable selectors and regularity of Hamiltonian
do not become obvious due to the dependency of action space on the state variable.
Interestingly, for our model we could show that the structure of drift and convexity of the running cost play
 in favour of our analysis and we can work with such state dependent action spaces. In particular, we obtain uniform
stability (Lemma~\ref{lem-stability}) and also show that the Hamiltonian is locally H\"{o}lder continuous (Lemma~\ref{lem-holder}, Lemma~\ref{lem-holder-0}). 
Since our action space depends on state variable we need to verify that the Filipov's criterion holds \cite[Chapter~18]{ali-bor}
and then by using Filipov's implicit function theorem we establish existence of a measurable minimizer for the Hamiltonian.  This is done
in Theorem~\ref{T-HJB}. But such a minimizer need not be continuous, in general, and one often requires a continuous minimizing
selector to construct $\eps$-optimal policies for the value functions $\widehat{V}^n$ (see \cite{atar-mandel-rei, ari-bis-pang}). 
With this in mind, we consider a perturbed problem where we perturb
the cost by a strictly convex function and show that the perturbed hamiltonian has a unique continuous selector (Lemma~\ref{lem-cont}).
In Theorem~\ref{T-HJB} below we show that this continuous minimizing selector is optimal for the perturbed ergodic control
problem and can be used to construct $\eps$-optimal policies (Theorem~\ref{T-upper}). To summarize our contribution in this 
paper, we have

\begin{itemize}
\item considered an ergodic control problem for the $M/M/N+M$ queueing network with  labor cross-training and identified the limit of the value functions,
\item solved the limiting HJB and established asymptotic optimality.
\end{itemize}

{\bf Notations:} By $\Rd$ we denote the $d$-dimensional Euclidean space equipped with the Euclidean norm $\abs{\cdot}$.
We denote by $\R^{d\times d}$ the set of all $d\times d$ real matrices and we endow this space with the usual metric.
For $a, b\in\R$, we denote the maximum (minimum) of $a$ and $b$  as $a\vee b$ ($a\wedge b$, respectively).
We define $a^+=a\vee 0$ and $a^-=- a\wedge 0$. $\lfl a\rfl$ denotes the largest integer that is small or equal to $a\in\R_+$.
Given a topological space $\mathcal{X}$ and $B\subset\mathcal{X}$,
the interior, closure, complement and boundary of $B$ in $\mathcal{X}$ are denoted by $B^0, \bar{B}, B^c$ and $\partial B$, respectively. $1_B$ is used to denote the characteristic function of the set $B$.
By $\calB(\mathcal{X})$ we denote the Borel $\sigma$-field of $\mathcal{X}$. Let $\calC([0, \infty) : \mathcal{X})$ 
be the set of all continuous functions from $[0, \infty)$ to $\mathcal{X}$. Given a path $f:\R_+\to\R$, we denote
by $\Delta f(t)$, jump of $f$ at time $t$, i.e., $\Delta f(t)= f(t)- f(t-)$.
 We define $\calC^k(\Rd), k\in\NN,$
as the set of all real valued $k$ times continuously differentiable functions on $\Rd$. For $\alpha\in(0,1)$,
$\calC^{k, \alpha}_{\mathrm{loc}}(\Rd)$ denotes the set of all real valued $k$-times continuously differentiable
 function on $\Rd$ with
its $k$-th derivative being locally $\alpha$-H\"{o}lder continuous on $\Rd$. 
For any any domain $D\subset\Rd$, $\Sob^{k, p}(D), p\geq 1,$ denotes the set of all $k$-times weakly differentiable functions 
that is in $L^p(D)$ and all its weak derivatives up to order $k$ are also in $L^p(D)$. By $\Sobl^{k, p}(D), p\geq 1,$ we denote the collection
of function that are $k$-times weakly differentiable and all its derivatives up to order $k$ are in $L^p_{\mathrm{loc}}(D)$.
 $\calC_{\text{pol}}(\Rd_+)$ denotes the
set of all real valued continuous functions $f$ that have at most polynomial growth i.e., 
$$\limsup_{|x|\to 0}\frac{\abs{f(x)}}{\abs{x}^k}= 0,\quad \text{for some }\; k\geq 0.$$ 
For a measurable $f$ and measure $\mu$ we denote $\langle f, \mu\rangle = \int f\, \D{\mu}$.
Let $\order(g)$ denote the space of function $f\in\calC(\Rd)$ such that
$$\sup_{x\in\Rd}\frac{\abs{f(x)}}{1+\abs{g(x)}}\;<\; \infty.$$
By $\sorder(g)$ we denote the subspace of $\order(g)$ containing function $f$ satisfying
$$\limsup_{\abs{x}\to\infty}\,\frac{\abs{f(x)}}{1+\abs{g(x)}}\;=\; 0.$$
Infimum over empty set is regarded as $+\infty$. $\kappa_1, \kap_2, \ldots,$ are deterministic positive constants whose value might change from line to line.

The organization of the paper is as follows. The next section introduces the setting of our model and state our main result 
on the convergence of the value functions. In Section~\ref{S-ergodic} we formulate the limiting controlled diffusion and state
our results on the ergodic control problem with state dependent action space. Section~\ref{S-spatial} obtains various results 
for the controlled diffusion and its HJB which are used to prove Theorem~\ref{T-HJB} from Section~\ref{S-ergodic}. Finally, in 
Section~\ref{S-optimality} we obtain asymptotic lower and upper bounds for the value functions.

\section{Setting and main result}\label{S-main}

Let $(\Omega,\mathfrak{F},\Prob)$ be a given complete probability 
space and all the
stochastic variables introduced below are defined on it.
The expectation w.r.t. $\Prob$ is denoted by $\Exp$. 
We consider a multiclass Markovian many-server queueing system which 
consists of $d$ customer classes and $d$ server pools.
Each server pool is assumed to contain $n$ identical servers (see Figure~\ref{plot}).

\begin{figure}
\centering
\def\svgwidth{0.6\columnwidth}
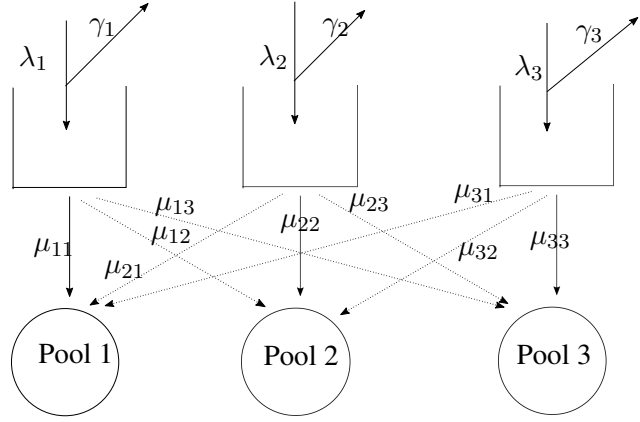
\caption{$d=3$. Dotted lines represents access of other classes to the pool in absence of its priority class customers
in the queue.}\label{plot}
\end{figure}

The system buffers are assumed to have infinite capacity.
Customers of class $i\in \{1,\dotsc,d\}$ arrive according to a Poisson process 
with rate $\lambda^{n}_{i}>0$. Upon arrival, customers enter the queue of their respective 
classes if not being processed. Customers of each class are 
served in the first-come-first-serve (FCFS) service discipline. 
Customers can abandon the system while waiting in the queue. 
Patience time of the customers are assumed to be exponentially distributed and class dependent.
Customers of class $i, i\in\{1, \ldots, d\},$ renege at  rate $\gamma^n_i$. We also assume that
no customer renege while in service. Customers of class $i$ have highest priority in accessing service 
from station $i$.
A customer of class $i, i\in\{1, \ldots, d\},$ is allowed to access service
from station $j, j\in\{1,\ldots, d\}, i\neq j,$ if and only if the $j$-th queue is empty and all the servers in the $i$-th pool are occupied by class-$i$ customers.
By $\mu^n_{ij}, i,j \in\{1, \ldots, d\},$ we denote the service rate of class $i$ at  station $j$. 
We denote $\mu^n_{ii}$ by $\mu^n_i$ for $i\in\{1, \ldots, d\}$.
We assume that customer arrivals, service and abandonment of all classes 
are mutually independent.

\smallskip
\emph{The Halfin-Whitt Regime.}
We study this queueing model in the Halfin-Whitt regime
(or the Quality-and-Efficiency-Driven (QED) regime).
We consider a sequence of systems indexed by $n$ where the arrival rates $\lam^n_i$ and $n$
grows to infinity at certain rates.
Let $\mathbf{r}^{n}_{i} \df \nicefrac{\lambda_{i}^{n}}{\mu^{n}_{i}}$
be the mean offered load of class $i$ customers.
In the Halfin-Whitt regime, the parameters are assumed to satisfy the following:
as $n\rightarrow \infty$,
\begin{equation}\label{HWpara}
\begin{gathered}
\frac{\lambda^{n}_{i}}{n} \;\to\;\lambda_{i}\;>\;0\,,\qquad
\mu_{i}^{n} \;\to\;\mu_{i}\;>\;0\,,\qquad \mu_{ij}^{n} \;\to\;\mu_{ij}\;\geq \;0,\,\, \text{for}\,\; i\neq j,\quad\\[5pt]
\gamma_{i}^{n} \;\to\;\gamma_{i}\;>\;0\, ,\,\quad
\frac{\lambda^{n}_{i} - n \lambda_{i}}{\sqrt n} \;\to\;\Hat{\lambda}_{i}\,,\qquad
{\sqrt n}\,(\mu^{n}_{i} - \mu_{i}) \;\to\;\Hat{\mu}_{i}\,,\\[5pt]
\frac{\mathbf{r}^{n}_{i}}{n} \;\to\;1 \;=\;
\frac{\lambda_{i}}{\mu_{i}}\, .
\end{gathered}
\end{equation}
We note that $\mu_{ij},\, i\neq j, $ could also be $0$ for some $i\neq j$. $\mu_{ij}=0$ could be understood as a situation where
servers at station $j$ are very inefficient in serving class-$i$ customers. 

\smallskip
\emph{State Descriptors.}
Let $X_{i}^{n} = \{X_{i}^{n}(t)\,\colon\,t\ge 0\}$ be the total number of class
$i$ customers in the system and $Q_{i}^{n} = \{Q_{i}^{n}(t)\,\colon\,t\ge 0\}$ be
the number of class $i$ customers in the queue. By
$Z_{ij}^{n}, i,\, j\in\{1, \ldots, d\},$ we denote
the number of class $i$ customers at the station $j$.
As earlier we denote $Z^n_{ii}$ by $Z^n_i$ for $i\in\{1, \ldots, d\}$.
The following basic relationships hold for these processes:
for each $t\ge 0$, and $i =1,\dotsc,d$,
\begin{equation}\label{rel1}
\begin{split}
&X_{i}^{n}(t)\;=\;Q_{i}^{n}(t) + Z^n_i + \sum_{j: j\neq i} Z_{ij}^{n}(t)\,,\\[5pt]
Q_{i}^{n}(t)\;\ge\;0\,,
\quad &Z_{ij}^{n}(t)\;\ge\;0\,,\quad \text{and} \quad Z^n_i + \sum_{k: k\neq i} Z_{ki}^{n}(t)\;\le\;n\,.
\end{split}
\end{equation}
Let
$\bigl\{A_{i}^{n}, S_{i}^{n}, S_{ij}^{n}, R_{i}^{n}\,,~i , j= 1,\dotsc,d\bigr\}$
be a collection of independent rate-$1$ Poisson processes.
Define
\begin{align*}
\Tilde{A}_{i}^{n}(t)&\;\df\;A_{i}^{n}(\lambda^{n}_{i} t)\,,\quad
\Tilde{S}_{i}^{n}(t)\;\df\;
S_{i}^{n}\biggl(\mu_{i}^{n} \int_{0}^{t} Z_{i}^{n}(s)\,\D{s}\biggr)\,,\\[5pt]
\Tilde{S}_{ij}^{n}(t)&\;\df\;
S_{ij}^{n}\biggl(\mu_{ij}^{n} \int_{0}^{t} Z_{ij}^{n}(s)\,\D{s}\biggr)\,,\quad
\Tilde{R}_{i}^{n}(t) \;\df\;
R_{i}^{n} \biggl(\gamma_{i}^{n} \int_{0}^{t} Q_{i}^{n}(s) \,\D{s}\biggr)\,.
\end{align*}
Then the dynamics takes the form
\begin{equation}\label{xp1}
X_{i}^{n}(t)\;=\;X_{i}^{n}(0) + \Tilde{A}_{i}^{n}(t)
- \Tilde{S}_{i}^{n}(t) -\sum_{j:j\neq i} \Tilde{S}_{ij}^n(t)- \Tilde{R}_{i}^{n}(t)\,,
\quad t\ge 0\,,~i = 1,\dotsc,d\,.
\end{equation}

\smallskip
\emph{Scheduling Control.}
We will consider policies that are non-anticipative. We also allow preemption.
Under these policies every customer class and its associated station
 must follow a work-conserving constrain in the following sense: for all $i\in\{1, \ldots, d\}$,
\begin{equation}\label{xp2}
(X^n_i(t)-n)^+ = Q^n_i(t) + \sum_{j: j\neq i}Z^n_{ij}(t), \quad t\geq 0.
\end{equation}
Combining \eqref{rel1} and \eqref{xp2} we see that
\begin{equation}\label{xp3}
Z^n_i = X^n_i \wedge n, \quad \i\in{1, \ldots, d}, \quad t\geq 0.
\end{equation}
Therefore, when a server from station $i$-becomes free and there are no customers of class-$i$ waiting in the queue,
 the server may process a customer of class $j, i\neq j$. Also a customer of class $i$
does not receive service from a server at the station $j, j\neq i,$ if there is an empty server at
station $i$. Service preemption is allowed,
i.e., service of a customer class can be interrupted at any time to serve some
other class of
customers and the original service is resumed at a later time, possibly by a server at some other station. It should be noted that a
policy need not be work-conserving. For instance, it could happen that under some policy there are empty servers at station
$j$ but there could be queue of class $i, \, i\neq j$.

Define the $\sigma$-fields as follows
\begin{align*}
\mathcal{F}^{n}_{t} &\;\df\;
\sigma \bigl\{ X^{n}(0),\;\Tilde{A}_{i}^{n}(t),\;
\Tilde{S}_{i}^{n}(t),\; \Tilde{S}^n_{ij}(t),\;\Tilde{R}_{i}^{n}(t)\,\colon\,
i =1,\dotsc,d\,,~ 0 \le s \le t\bigr\} \vee \mathcal{N}\,,
\intertext{and}
\mathcal{G}^{n}_{t} &\;\df\;
\sigma \bigl\{\delta\Tilde{A}_{i}^{n}(t,r),\;\delta\Tilde{S}_{i}^{n}(t,r),\; \del\Tilde{S}^n_{ij}(t),\;
\delta\Tilde{R}_{i}^{n}(t,r)\,\colon\,i =1,\dotsc,d\,,~ r\ge 0 \bigr\}\,,
\end{align*}
where
\begin{align*}
\delta\Tilde{A}_{i}^{n}(t,r)&\;\df\; \tilde{A}_{i}^{n}(t+r)
-\tilde{A}_{i}^{n}(t)\,,\quad
\delta\Tilde{S}_{i}^{n}(t,r)\;\df\;
S_{i}^{n}\biggl(\mu_{i}^{n}\int_{0}^{t} Z_{i}^{n}(s)\,\D{s}+\mu_{i}^{n}\, r\biggr)
-\Tilde{S}_{i}^{n}(t)\,,
\\[5pt]
\delta\Tilde{S}_{ij}^{n}(t,r) &\;\df\;
S_{ij}^{n}\biggl(\mu_{ij}^{n}\int_{0}^{t} Z_{ij}^{n}(s)\,\D{s}+\mu_{ij}^{n}\, r\biggr)
-\Tilde{S}_{i}^{n}(t)\,,\\[5pt]
\delta\Tilde{R}_{i}^{n}(t,r) &\;\df\;
R_{i}^{n}\biggl(\gamma_{i}^{n}\int_{0}^{t} Q_{i}^{n}(s)\,\D{s}
+\gamma_{i}^{n}r\biggr)-\Tilde{R}_{i}^{n}(t)\,,
\end{align*}
and $\mathcal{N}$ is the collection of all $\Prob$-null sets.
The filtration $\{\mathcal{F}^{n}_{t}\,,\;t\ge0\}$ represents the information
available up to time $t$ while $\mathcal{G}^{n}_{t}$ contains
the information about future increments of the processes.

We say that a  control policy is \emph{admissible} if it satisfies \eqref{xp2} (or, equivalently \eqref{xp3}) and,
\begin{itemize}
\item[(i)]
$Z^{n}(t)$ is adapted to $\calF^{n}_{t}$,
\smallskip
\item[(ii)]
$\mathcal{F}^{n}_{t}$ is independent of $\mathcal{G}^{n}_{t}$ at each time
$t\ge 0$,
\smallskip
\item[(iii)]
for each $i=1,\dotsc,d$, and $t\ge 0$, the process
$\delta\Tilde{S}_{i}^{n}(t,\cdot\,)$ $\Bigl(\delta\Tilde{S}_{ij}^{n}(t,\cdot\,)\Bigr)$
agrees in law with $S_{i}^{n}(\mu_{i}^{n}\,\cdot\,)$ $\Bigl(S_{ij}^{n}(\mu_{ij}^{n}\,\cdot\,)\Bigr)$, and the process
$\delta\Tilde{R}_{i}^{n}(t,\cdot\,)$
agrees in law with $R_{i}^{n}(\gamma_{i}^{n}\,\cdot\,)$.
\end{itemize}
By criterion (iii) above the increments of the processes have same distribution as the original processes in addition to being independent 
of $\mathcal{F}^N_t$ (by (ii) above).
We denote the set of all admissible control policies
$(Z^n, \mathcal{F}^{n}, \mathcal{G}^{n})$ by $\Uadm^{n}$. 

\subsection{Control problem formulation} 
\label{S-diffcontrol}

Define the diffusion-scaled processes
$$\Hat{X}^{n} = (\Hat{X}^{n}_{1},\dotsc,\Hat{X}^{n}_{d})\transp\,,\quad
\Hat{Q}^{n} = (\Hat{Q}^{n}_{1},\dotsc,\Hat{Q}^{n}_{d})\transp\,,
~\text{and}\quad
\Hat{Z}^{n} = \Bigl[\Hat{Z}^{n}_{ij}\Bigr]\,,$$
by
\begin{equation}\label{dc1}
\begin{split}
\Hat{X}^{n}_{i}(t) &\;\df\;\frac{1}{\sqrt{n}}(X^{n}_{i}(t)-  n )\,,\\[5pt]
\Hat{Q}^{n}_{i}(t) &\;\df\;\frac{1}{\sqrt{n}} Q^{n}_{i}(t)\,,\quad 
\Hat{Z}^n_{ij}\; \df\; \frac{1}{\sqrt{n}} Z^n_{ij}, \; i\neq j,\\[5pt]
\Hat{Z}^{n}_{i}(t) &\;\df\;\frac{1}{\sqrt{n}}(Z^{n}_{i}(t)-  n )
\end{split}
\end{equation}
for $t\ge 0$.
By \eqref{xp1} and \eqref{xp3}, we can express $\Hat{X}^{n}_{i}$ as
\begin{align}\label{dc2}
\Hat{X}^{n}_{i}(t) &\;=\;\Hat{X}^{n}_{i}(0) +\ell_{i}^{n} t 
+ \mu_{i}^{n} \int_{0}^{t} (\Hat{X}_{i}^{n})^-(s)\,\D{s}-\sum_{j:j\neq i}\mu^n_{ij}\int_0^t\Hat{Z}^n_{ij}(s)\, \D{s}
- \gamma_{i}^{n} \int_{0}^{t} \Hat{Q}^{n}_{i}(s)\,\D{s}\\[5pt]
&\mspace{20mu} + \Hat{M}_{A,i}^{n}(t) - \Hat{M}_{S,i}^{n}(t)-\sum_{j:j\neq i} \Hat{M}_{S, ij}^n(t)
- \Hat{M}_{R,i}^{n}(t)\,,\nonumber
\end{align}
where $\ell^{n}= (\ell_{1}^{n},\dotsc,\ell_{d}^{n})\transp$ is defined as
\begin{equation*}
\ell_{i}^{n} \;\df\;\frac{1}{\sqrt{n}}(\lambda_{i}^{n}
- \mu_{i}^{n} n)= \frac{\lam^n_i-n\lam_i}{\sqrt{n}}-\sqrt{n}(\mu^n_i-\mu_i)\,,
\end{equation*}
and
\begin{equation}\label{dc3}
\begin{split}
\Hat{M}_{A,i}^{n}(t) &\;\df\;\frac{1}{\sqrt{n}}(A_{i}^{n}(\lambda^{n}_{i} t )
- \lambda^{n}_{i} t)\,,\\[5pt]
\Hat{M}_{S,i}^{n}(t) &\;\df\;\frac{1}{\sqrt{n}} 
\left( S_{i}^{n}\left(\mu_{i}^{n} \int_{0}^{t} Z_{i}^{n}(s)\,\D{s}\right) 
- \mu_{i}^{n} \int_{0}^{t} Z_{i}^{n}(s)\,\D{s} \right)\,,\\[5pt]
\Hat{M}_{S,ij}^{n}(t) &\;\df\;\frac{1}{\sqrt{n}} 
\left( S_{ij}^{n}\left(\mu_{i}^{n} \int_{0}^{t} Z_{ij}^{n}(s)\,\D{s}\right) 
- \mu_{i}^{n} \int_{0}^{t} Z_{ij}^{n}(s)\,\D{s} \right)\,,\\[5pt]
\Hat{M}_{R,i}^{n}(t) &\;\df\;\frac{1}{\sqrt{n}} 
\left( R_{i}^{n} \left(\gamma_{i}^{n} \int_{0}^{t} Q_{i}^{n}(s) \,\D{s} \right)
- \gamma_{i}^{n} \int_{0}^{t} Q_{i}^{n}(s) \,\D{s} \right),
\end{split}
\end{equation}
are square integrable martingales w.r.t. the filtration $\{\calF^{n}_{t}\}$.

Note that by \eqref{HWpara}
$$\ell_{i}^{n} \;=\;
\;\xrightarrow[n\to\infty]{}\;\ell_{i}\;\df\;
\Hat{\lambda}_{i}-  \Hat{\mu}_{i}\,.$$

By $\R^{d\times d}_+$ we denote the set of real matrices with non-negative entries. Define
$$\M \;\df\;\bigl\{m \in \R^{d\times d}_+\,\colon\,u_{ii}=0, \, \sum_{k:k\neq i}u_{ki}\;\leq\; 1,\; \forall\; i\bigr\}\,.$$ 
For any $x\in\Rd$, we define
\begin{equation}\label{con-set}
\M(x)\;\df\; \{u\in\M : \sum_{j: j\neq i}m_{ij}\, x_j^- \;\leq\; x^+_i,\; \forall\; i\}.
\end{equation}
It is easy to see that $\M(x)$ is a non-empty convex and compact subset of $\M$ for all $x\in\Rd$. Also
$0\in\M(x)$, for all $x\in\Rd$.
From \eqref{rel1}, \eqref{xp2} and \eqref{xp3} we have for $i\in\{1, \ldots, d\}$,
\begin{equation}\label{01}
\begin{split}
\sum_{k: k\neq i} \Hat{Z}^n_{ki}& \;\leq\; (\Hat{X}^n_i)^-,
\\[5pt]
\sum_{j:j\neq i} \Hat{Z}^n_{ij} &\;\leq \; (\Hat{X}^n_i)^+.
\end{split}
\end{equation}
Define
$$\Hat{U}^n_{ki}(t)\; \df\; \frac{\Hat{Z}^n_{ki}(t)}{(\Hat{X}^n_i(t))^-},\quad t\geq 0\,,$$
where we fix $\Hat{U}^n_{ki}(t) = 0$ if $(\Hat{X}^n_i(t))^-=0$.
We also set $\Hat{U}^n_{ii}(t) = 0$, for all $i\in\{1, \ldots, d\}$, and $t\geq 0$.
Therefore using \eqref{01} we obtain,
\begin{equation}\label{02}
\begin{split}
\sum_{k: k\neq i} \Hat{U}^n_{ki}(t)& \;\leq\; 1,
\\
\sum_{j:j\neq i} \Hat{U}^n_{ij}(t)\, (\Hat{X}^n_j(t))^-&\;\leq \; (\Hat{X}^n_i(t))^+.
\end{split}
\end{equation}
Thus $\Hat{U}^n(t)\in\M(\Hat{X}^n(t))$ for all $t$ and $\Hat{U}^n(t)$ is $\mathcal{F}^n_t$ adapted.
Also $\Hat{U}^{n}_{ij}$ represents the fraction of the number of servers $(X^n_j-n)^-$ at station $j$ that are serving
class-$i$ customers.
As we show later, it is convenient to view $\Hat{U}^{n}(t)$ as the control.

\subsubsection{The cost minimization problem}

We next introduce the running cost function for the control problem.
Let $r\colon \Rd_{+} \rightarrow \R_{+}$ be a given function satisfying 
\begin{equation}\label{cost1}
0 \;\le\;r(x) \;\le\;c_{1}(1+\abs{x}^{m})
\qquad\text{for some}~m \ge 1\,,
\end{equation}
and some positive constant $c_1$. 
We also assume that $r$ is convex and therefore, locally Lipschitz.
 For example,
if we let $h_{i}, h_i\geq 0,$ be the holding cost rate for class $i$ customers, then some 
of the typical running cost functions are the following:
\begin{equation*}
r(x) \;=\;\sum_{i=1}^{d} h_{i} x_{i}^{m_i}\,,\quad \min_{i}\, m_i\ge 1\,.
\end{equation*}
These running cost functions evidently satisfy the condition in \eqref{cost1}. 

Given the initial state $X^{n}(0)$ and an admissible scheduling 
policy $Z^{n} \in \Uadm^{n}$, we 
define the diffusion-scaled cost function as 
\begin{equation}\label{costd1}
J(\Hat{X}^{n}(0),\Hat{Z}^{n}) \;\df\;\limsup_{T \to \infty}\;
\frac{1}{T}\,\Exp\biggl[ \int_{0}^{T} r(\Hat{Q}^{n}(s)) \,\D{s} \biggr]\,,
\end{equation}
where the running cost function $r$ satisfies \eqref{cost1}.
Then, the associated cost minimization problem is defined by
\begin{equation}\label{E-Vn}
\widehat{V}^{n}(\Hat{X}^{n}(0)) \;\df\;\inf_{Z^{n}\in \Uadm^{n}}\;
J(\Hat{X}^{n}(0),\Hat{Z}^{n})\,.
\end{equation}

We refer to $\widehat{V}^{n}(\Hat{X}^{n}(0))$ as the
\emph{diffusion-scaled value function} given 
the initial state $\Hat{X}^{n}(0)$ for the $n^{\rm th}$ system. 

From \eqref{xp2} we see that for $i\in\{1, \ldots, d\}$, and $t\geq 0$,
$$\Hat{Q}^n_i(t)\;=\;(\Hat{X}^n_i(t))^+ - \sum_{j: j\neq i}\Hat{U}^n_{ij}(t)\, (\Hat{X}^n_j(t))^-.$$
Therefore redefining $r$ as 
\begin{equation}\label{05}
r(x,u) \;=\; r(x_1^+-\sum_{j:j\neq 1} u_{1j}x_j^-, \ldots, x_d^+-\sum_{j:j\neq d} u_{dj}x_j^-), \quad u\in\M(x),
\end{equation}
we can rewrite the control problem as 
$$\widehat{V}^{n}(\Hat{X}^{n}(0))\;=\;\inf\;\Tilde{J}(\Hat{X}^{n}(0),
\Hat{U}^{n})\,,$$
where
\begin{equation}\label{cost34}
\Tilde{J}(\Hat{X}^{n}(0),\Hat{U}^{n})\;\df\; \limsup_{T\to\infty}\;
\frac{1}{T}\,\Exp \biggl[ \int_{0}^{T} 
r\bigl(\Hat{X}^{n}(s),\Hat{U}^{n}(s)\bigr) \,\D{s} \biggr]\,,
\end{equation}
and the infimum is taken over all admissible pairs $(\Hat{X}^{n},\Hat{U}^{n})$ 
satisfying \eqref{02}. Hence $\Hat{U}^n(t)\in\M(\Hat{X}^n(t))$, almost surely,  for all $t\geq 0$.

For simplicity we assume that the initial condition $\Hat{X}^{n}(0)$ 
is deterministic and $\Hat{X}^{n}(0)\to x$, as $n\to\infty$, for some $x\in\Rd$.

\subsubsection{The limiting controlled diffusion process}
As in \cite{atar-mandel-rei,harrison-zeevi, ari-bis-pang}, the analysis will be done by studying the limiting controlled 
diffusions.
One formally deduces that, 
provided $\Hat{X}^{n}(0) \to x$, there exists a limit $X$ for $\Hat{X}^{n}$ on
every finite time interval, and the limit process $X$ is a
$d$-dimensional diffusion process, that is,
\begin{equation}\label{dc4}
\D X_{t}\;=\;b(X_{t},U_{t}) \,\D{t} + \Sigma\,\D W_{t}\,,
\end{equation}
with initial condition $X_{0}\;=\;x$.
In \eqref{dc4} the drift
$b(x,u)\colon\Rd\times\M \rightarrow \Rd$ takes the form
\begin{equation}\label{dc5}
b_i(x,u)\;=\;\ell_i + \mu_i x_i^- - \sum_{j: j\neq i}\mu_{ij}u_{ij} x_j^{-} -\gamma_i \Bigl(x_i^+ - \sum_{j:j\neq i} u_{ij} x_j^-\Bigr)\,,
\end{equation}
with
\begin{align*}
\ell &\;\df\;(\ell_{1},\dotsc,\ell_{d}) \transp\,.
\end{align*}
The  control $U_{t}$ lives in $\M$ and is non-anticipative,
$W$ is a $d$-dimensional standard Wiener process independent of 
the initial condition $X_{0}=x$, and the covariance matrix is given by
$$\Sigma \Sigma\transp\;=\;\diag\,(2 \lambda_{1},\dotsc,2 \lambda_{d})\,.$$ 
A formal derivation of the drift in \eqref{dc5} can be obtained from 
\eqref{xp2} and \eqref{dc2} . We also need the control to satisfy $U_t\in\M(X(t))$ for all $t\geq 0$.
We define $q:\Rd\times\M\to\Rd$ as follows,
\begin{equation}\label{0000}
q_i(x, u) = x_i^+ - \sum_{j:j\neq i} u_{ij} x_j^-, \quad i\in\{1, \ldots, d\}.
\end{equation}
Thus from \eqref{dc5} we get that
\begin{equation}\label{dc6}
b_i(x, u)\;=\;\ell_i + \mu_i x_i^- - \sum_{j: j\neq i}\mu_{ij}u_{ij} x_j^{-} -\gamma_i q_i(x, u).
\end{equation}
A detailed description of equation \eqref{dc4} and related results 
are given in Section~\ref{S-ergodic}.

\subsubsection{The ergodic control problem for controlled diffusion} 
Define $\Tilde{r}\colon \Rd_{+} \times \M \rightarrow \R_{+}$, by
$$\Tilde{r}(x,u)\;\df\;r(q_1^+(x, u), \ldots, q_d^+(x, u))\,.$$
We note that for $u\in\M(x)$ the cost $\Tilde{r}(x, u)$ agrees with $r(x, u)$ given by \eqref{05}.
In analogy with \eqref{cost34} we define the ergodic cost associated 
with the controlled diffusion process $X$ and the running cost function 
$\Tilde{r}(x,u)$ as
\begin{equation*}
J(x,U)\;\df\;\limsup_{T\rightarrow \infty}\;\frac{1}{T}\,
\Exp^{U}_{x}\biggl[\int_{0}^{T} \Tilde{r}(X_{t},U_{t}) \,\D{t}\biggr]\,,
\quad  U \in \Uadm\,.
\end{equation*}
Here $\Uadm$ denotes set of all admissible controls which are defined in Section~\ref{S-ergodic}.
We consider the ergodic control problem 
\begin{equation}\label{dcp2}
\varrho_{*}(x)\;=\;\inf_{U \in \Uadm}\;J(x,U)\,.
\end{equation}
We call $\varrho_{*}(x)$ the optimal
value at the initial state $x$ for the
controlled diffusion process $X$.
It is shown later that $\varrho_{*}(x)$ is independent of $x$.
A detailed treatment
and related results corresponding to the ergodic control problem are given
in Section~\ref{S-ergodic}. 

Next we state the main result of this section, the proof of which can be
found in Section~\ref{S-optimality}.

\begin{theorem}\label{T-optimality}
Let $\Hat{X}^{n}(0)\to x\in\Rd$, as $n\to\infty$.
Also assume that \eqref{HWpara} and \eqref{cost1} hold where the cost function $r$ is convex.
Then
$$\lim_{n\to\infty}\;\widehat{V}^{n}(\Hat{X}^{n}(0))\; = \;\varrho_{*}\,,$$
where $\varrho_{*}(x)$ is given by \eqref{dcp2}.
\end{theorem}
Theorem~\ref{T-optimality} is similar to Theorems~2.1 and 2.2 in \cite{ari-bis-pang}. The central idea of the proof of 
Theorem~\ref{T-optimality} is same as that of \cite{ari-bis-pang}. One of the main advantage of the present setting is
the stability of the system. We could directly show that for all $n$ large the mean-empirical measures corresponding
to $n$-th system has all polynomial moments finite under every admissible policy 
(see Lemma~\ref{lem-queue-stability}). This is not the case in 
\cite{ari-bis-pang} where a spatial truncation method is used to treat such difficulty.
We must also note that the action space in our setting depends on the location $x$ whereas in \cite{ari-bis-pang} 
the action space is a fixed compact set. Therefore we need to adopt suitable modification for this problem.
 As shown below the convexity property of the cost and the structure of drift $b$ play a key role in our analysis.

\section{An Ergodic Control Problem for Controlled Diffusions}
\label{S-ergodic}

\subsection{The controlled diffusion model}

The dynamics are modeled by 
controlled diffusion processes $X = \{X(t),\;t\ge0\}$
taking values in  $\Rd$, and
governed by the It\^o stochastic differential equation
\begin{equation}\label{E-sde}
\D{X}(t) \;=\;b(X(t),U(t))\,\D{t} + \Sigma \,\D{W}(t)\,,
\end{equation}
where $b$ is given by \eqref{dc6} and 
\begin{align*}
\Sigma\Sigma^T \; =\; \text{diag}\,(2\lam_1, \ldots, 2\lam_d).
\end{align*}
All random processes in \eqref{E-sde} live in a complete
probability space $(\Omega,\sF,\{\sF_t\},\Prob)$.
The process $W$ is a $d$-dimensional standard Brownian motion independent
of the initial condition $X_{0}$.

\begin{definition}\label{admissible}
A process $U$, taking values in $\M$ and
$U(t)(\omega)$ is jointly measurable in
$(t,\omega)\in[0,\infty)\times\Omega$, is said to be an admissible control if,
there exists a strong solution $X=\{X(t)\, :\, t\geq 0\}$ satisfying \eqref{E-sde}, and, 
\begin{itemize}
\item $U$ is \emph{non-anticipative}:
for $s < t$, $W_{t} - W_{s}$ is independent of
\begin{equation*}
\sF_{s} \;\df\;\text{the completion of~} \sigma\{X_{0},U(p),W(p),\;p\le s\}
\text{~relative to~}(\sF,\Prob)\,.
\end{equation*}
\item $U(t)\in\M(X(t))$, almost surely, for  $t\geq 0$.
\end{itemize}
\end{definition}
We let $\Uadm$ denote the set of all admissible controls.
Note that the drift $b$ is Lipschitz continuous
and the diffusion matrix $\Sigma$ non-degenerate. 
Since $0\in\M(x)$ for all $x\in\Rd$ we see that $U\equiv 0$ is in $\Uadm$. Thus $\Uadm$ is non-empty.
Let $a\;\df\; \Sigma\Sigma^T$. We define the family of operators
$L^{u}\colon\calC^{2}(\Rd)\to\calC(\Rd)$,
with parameter $u\in\M$, by
\begin{equation}\label{E3.3}
L^{u} f(x) \;\df\;\frac{1}{2}\sum_{i,j=1}^d a_{ij}\,\partial_{ij} f(x)
+ \sum_{i=1}^d b_{i}(x,u)\, \partial_{i} f(x)\,,\quad u\in\M\,.
\end{equation}
We refer to $L^{u}$ as the \emph{controlled extended generator} of
the diffusion \eqref{E-sde}.
In \eqref{E3.3} and elsewhere in this paper we have adopted
the notation $\partial_{i}\df\tfrac{\partial~}{\partial{x}_{i}}$ and
$\partial_{ij}\df\tfrac{\partial^{2}~}{\partial{x}_{i}\partial{x}_{j}}$.

A control $U\in\Uadm$ is said to be stationary Markov if for some measurable $v:\Rd\to\M$ we have
$U(t)=v(X(t))$. Therefore for a stationary Markov control we have $v(X(t))\in\M(X(t))$ for all $t\geq 0$.
By $\Usm$ we denote the set of all stationary Markov controls. 

Now we introduce relaxed controls which will
be useful for our analysis. Association of  relaxed controls  to such control problems is useful since it extends
the action space to a compact, convex set \cite{ari-bor-ghosh}.
In our setup we show that we can not do better even we extend the controls to include relaxed controls
(see Theorem~\ref{T-HJB} and \ref{T-exist}).
Moreover, relaxed control would be useful to prove asymptotic lower bounds (Theorem~\ref{T-lower}).
By $\mathcal{P}(\M)$ we denote the set of all probability measures on $\M$. We can extend 
the drift $b$ and the running cost $r$ on $\mathcal{P}(\M)$ as follows: for $v(\D{u})\in\mathcal{P}(\M)$,
\begin{equation*}
b(x,v)\;\df\;\int_\M b(x,u)v(\D{u})\,,\quad \text{and}\quad
r(x,v)\;\df\;\int_\M r(x,u)v(\D{u})\,.
\end{equation*}
Controls taking values in $\mathcal{P}(\M)$ are known as \textit{relaxed} controls. 
Similarly, we can extend the definition of stationary Markov controls to measure valued processes.
A stationary Markov control $v:\Rd\to\mathcal{P}(\M)$ is said to be admissible if there  is a unique strong solution
to \eqref{E-sde} and $\langle 1_{M(X(t))}, v(X(t))\rangle =1$, almost surely, for all $t\geq 0$.
We continue to denote this extended class by $\Usm$. We endow the space $\Usm$ with
 the following topology \cite[Chapter~2.2.4]{ari-bor-ghosh}:
: $v_{n}\to v$ in $\Usm$ if and only if
$$\int_{\Rd}f(x)\int_\M g(x,u)v_{n}(\D{u}\mid x)\,\D{x}
\;\xrightarrow[n\to\infty]{}\;
\int_{\Rd}f(x)\int_\M g(x,u)v(\D{u}\mid x)\,\D{x}$$
for all $f\in L^{1}(\Rd)\cap L^{2}(\Rd)$
and $g\in\calC_{b}(\Rd\times \M)$.
\begin{proposition}\label{compact}
The space $\Usm$ under the above mentioned topology is compact whenever the initial condition $X(0)$ is a fixed point.
\end{proposition}

\begin{proof}
Let $X(0)=x$. Let $\{v_n\}\in\Usm$ be a sequence of stationary Markov controls. Then $(X_n(t), v_n(X_n(t)))$ satisfies \eqref{E-sde} and
\begin{equation}\label{10}
\langle 1_{M(X_n(t))}, v_n(X_n(t))\rangle =1, \quad a.s., \quad \text{for}\;\; t\geq 0.
\end{equation}
Now from \cite[Section~2.4]{ari-bor-ghosh} there exists a measurable $v:\Rd\to\mathcal{P}(\M)$ such that $v_n\to v$ as
$n\to\infty$ along some subsequence in the topology of Markov control. We continue to denote the subsequence by $v_n$. Now consider the strong solution
$X(t)$ corresponding to the Markov control $v$. Existence of the unique strong solution is assured by \cite[Theorem~2.2.12]{ari-bor-ghosh}.
Moreover, $X_n\Rightarrow X$ as $n\to\infty,$ in $\calC([0, \infty), \Rd)$. A similar the argument as in \cite[Lemma~2.4.2]{ari-bor-ghosh}
shows that for any $t>0$, 
\begin{equation}\label{11}
\norm{p_n(t, x, \cdot)-p(t, x, \cdot)}_{L^1(\Rd)}\to 0, \quad as\quad n\to\infty, 
\end{equation}
where $p_n, p$ denote the transition density of $X_n,\, X$ respectively, at time $t$ starting from $x$.
Also the transition densities are locally bounded.
 Now observe that
$(x, u)\mapsto 1_{\M^c(x)}(u)$ is a lower-semicontinuous function. This fact follows from the definition of $\M(x)$.
Hence there exists a sequence of bounded, continuous function $g_k : \Rd\times\M\to \R$ such that
(see for example, \cite[Proposition~2.1.2]{krantz})
$$g_k(x, u)\nearrow 1_{M^c(x)}(u), \quad \text{pointwise}, \quad \text{as}\; \; k\to\infty.$$
Let $\phi_k$ be any smooth, non-negative function taking values in $[0, 1]$ with support in $\B_k(0)$.
We choose $\phi_k$ to satisfy $\phi_k\nearrow 1$, as $k\to\infty$. Then from the convergence criterion
of Markov controls we get for all $k\geq 1$,
\begin{align*}
&\int_{\Rd}p(t, x, y)\phi_k(y)\int_{\M}g_{k}(y, u)v(\D{u} | y) dy
\\[5pt]
 &\leq \int_{\Rd}p(t, x, y)\phi_k(y)\int_{\M}g_{k}(y, u)\bigl(v(\D{u} | y)\, \D{y}
-v_n(\D{u} | y)\bigr)
\\[5pt]
 &\, \hspace{4mm}\ + \sup_{(y, u)\in\Rd\times\M}\abs{g_k(y, u)}\norm{p_n(t, x, \cdot)-p(t, x, \cdot)}_{L^1(\Rd)}
 \\[5pt]
 &\,\hspace{4mm}\, + \int_{\Rd}p_n(t, x, y)\phi_k(y)\int_{\M}g_{k}(y, u) v_n(\D{u} | y)
\\[5pt]
&\leq \int_{\Rd}p(t, x, y)\phi_k(y)\int_{\M}g_{k}(y, u)\bigl(v(\D{u} | y)
-v_n(\D{u} | y)\bigr)
\\[5pt]
 &\, \hspace{4mm}\ + \sup_{(y, u)\in\Rd\times\M}\abs{g_k(y, u)}\norm{p_n(t, x, \cdot)-p(t, x, \cdot)}_{L^1(\Rd)}
 \\[5pt]
 &\,\hspace{4mm}\, + \int_{\Rd}p_n(t, x, y)\int_{\M} 1_{\M^c(y)}(u)v_n(\D{u} | y)
 \\[5pt]
 &\longrightarrow 0, \quad \text{as}\quad n\to\infty,
\end{align*}
where in the last line we use \eqref{10}, \eqref{11}. Therefore letting $k\to\infty$ we have for $t>0$ that
$$\Exp\Bigl[\int_{\M}1_{\M^c(X(t))}v(\D{u} | X(t))\Bigr]=0.$$
This proves that for every $t>0$, support of $v(\D{u}| X(t))$ lies in $M(X(t))$ almost surely. Now define $v(x)=0$ 
to make sure $v(X(0))\in M(X(0))$. This completes
the proof.
\end{proof}
It is well known \cite[Theorem~2.2.12]{ari-bor-ghosh} that for every Markov control in $\Usm$ there is a unique strong solution to \eqref{E-sde} which is also a strong
Markov process. A stationary control is called stable if the associated Markov process $X$ is positive recurrent. 

We recall the cost function $\tilde{r}$ from previous section where
$$\Tilde{r}(x, u)= r(q^+_1(x, u), \ldots, q^+_d(x, u))$$
and $r$ is a convex function that satisfies \eqref{cost1}. Define
$$\bvnorm{u}\;\df\; \sum_{i, j=1}^d \abs{u_{ij}}^2, \quad \text{where}\quad u\in\M.$$
For $\eps\in(0, 1),$ we consider the following perturbed cost function
$$\tilde{r}_\eps(x, u)\; \df\; \Tilde{r}(x, u) + \eps\bvnorm{u}.$$
Since $\bvnorm{\cdot}$ is strictly convex in $u$, we have $\Tilde{r}_\eps(x, \cdot)$ strictly convex on $\M(x)$ for every $x\in\Rd$. 
For $U\in\Uadm$, we define
\begin{align*}
J(x, U)&\;=\limsup_{T\to\infty}\frac{1}{T}\Exp\Bigl[\int_0^T \Tilde{r}(X(s), U(s))\D{s}\Bigr],
\\
J_\eps(x, U)&\;=\limsup_{T\to\infty}\frac{1}{T}\Exp\Bigl[\int_0^T \Tilde{r}_\eps(X(s), U(s))\D{s}\Bigr].
\end{align*}
Therefore we have two value functions given by
\begin{equation}\label{0002}
\varrho_* \;\df\; \inf_{U\in\Uadm}J(x, U), \quad \varrho_\eps\;\df\; \inf_{U\in\Uadm}J_\eps(x, U).
\end{equation}
We have suppressed the dependency of $x$ from the value function as it is shown below that the value functions do not depend on $x$.
Our main result of this section is the following.

\begin{theorem}\label{T-HJB}
There exist $V_\eps,\, V\in\calC^2(\Rd)$ satisfying the following Hamilton-Jacobi-Bellman (HJB) equations,
\begin{align}
\min_{u\in\M(x)} \Bigl(L^u V_\eps(x) + \Tilde{r}_\eps(x, u)\Bigr) &\;=\; \varrho_\eps, \quad \text{for}\;\; \eps\in(0, 1),\label{eps-HJB}
\\
\min_{u\in\M(x)} \Bigl(L^u V(x) + \Tilde{r}(x, u)\Bigr)&\;=\; \varrho_{*},\label{HJB}
\end{align}
such that
\begin{enumerate}
\item[(i)] for each $\eps\in(0, 1)$, there is a unique continuous selector of \eqref{eps-HJB} which is also optimal for $\varrho_\eps$;
\item[(ii)] if a pair $(\bar{V}_\eps, \bar\varrho_\eps)\in\calC^2(\Rd)\cap\calC_{\mathrm{pol}}(\Rd)\times\R$ 
satisfies \eqref{eps-HJB} then we have $(\bar{V}_\eps, \bar\varrho_\eps)=({V}_\eps, \varrho_\eps)$;
\item[(iii)] as $\eps\to 0$, $V_\eps\to V$ in $\Sobl^{2, p}(\Rd),\, p\geq 1$;
\item[(iv)] every measurable selector of \eqref{HJB} is optimal for $\varrho_{*}$;
\item[(v)] if a pair $(\bar{V}, \bar\varrho)\in\calC^2(\Rd)\cap\calC_{\mathrm{pol}}(\Rd)\times\R$ 
satisfies \eqref{HJB} then we have $(\bar{V}, \bar\varrho)=({V}, \varrho_{*})$.
\end{enumerate}
\end{theorem}

It is not hard to see that $\varrho_\eps\searrow\varrho_{*}$, as $\eps\to 0$ (the difference between the running costs
is the order of $\eps$).
Therefore using Theorem~\ref{T-HJB} we can find $\eps$-optimal
controls (in fact, continuous Markov controls) for $\varrho_{*}$. Continuity property of the Markov controls plays a key role in 
the construction of  $\eps$-optimal admissible controls for the queueing models and in obtaining asymptotic upper bound for Theorem~\ref{T-optimality}.
Results similar to Theorem~\ref{T-HJB} are also obtained in \cite{ari-bis-pang} for a fixed action space  that does not depend on $x$. 
Therefore the results of \cite{ari-bis-pang} do not directly apply  here.
Because of the state dependency of the action space we need to put extra effort to establish regularity properties of the value function.
Also, finding a measurable minimizing selector of \eqref{HJB} becomes less obvious.
On the other hand, we have
uniform stability (see Lemma~\ref{lem-stability}), an advantage compared to \cite{ari-bis-pang}, which help us in proving Thoerem~\ref{T-HJB}.
It is shown that the $\eps$-perturbed Hamiltonian
in \eqref{eps-HJB} has certain regularity properties (Lemma~\ref{lem-holder}) which together with the strict convexity of the
perturbed cost $\Tilde{r}_\eps$ help us in finding a unique continuous minimizing selector (Lemma~\ref{lem-cont}).
Using these properties we characterize  the discounted value function $V^\al_\eps$ with the
running cost $\Tilde{r}_\eps$ in Theorem~\ref{T-discount}. Then using uniform stability and the Sobolev estimates
we show that the scaled value function $\bar{V}^\al_\eps\;\df\; V^\al_\eps-V^\al_\eps(0)$ converges to 
a limit $V_\eps$, as $\alpha\to 0$, that solves \eqref{eps-HJB}. A similar argument is also used to justify the passage of limit in $V_\eps\to V$ as $\eps\to 0$.
The detailed proof of  Theorem~\ref{T-HJB} is given in section~\ref{S-spatial}.

\section{Uniform stability and related results}\label{S-spatial}
The goal of this section is to prove Theorem~\ref{T-HJB} and obtain related estimates. 
In what follows, we use several standard results from the theory of elliptic PDE's without explicitly mentioning its reference. 
Those results can be found in \cite{gil-tru}, \cite[Appendix]{ari-bor-ghosh}. For instance, the
following is used in several places:
if a function $\psi\in\Sobl^{2, p}(\B), p>d, \B$ is an open set, satisfies 
$\frac{1}{2}\sum_{ij}a_{ij} \partial_{ij}\psi(x) = f(x)$ for some
function $f\in \calC^\al(\B), \al\in(0, 1)$, then $\psi\in\calC_{\mathrm{loc}}^{2, \beta}(\B)$ for some $\beta>0$.

We start by proving continuity property of
the selector. Define a map $\varphi: \Rd\times\Rd\to \M$ as follows.
$$\varphi(x, p)\; \df\; \Argmin_{u\in\M(x)} \{b(x, u)\cdot p + \Tilde{r}_\eps(x, u)\}.$$
Since $\M(x)$ is convex and compact for every $x$ and $\Tilde{r}_\eps(x, \cdot)$ is strictly 
convex on $\M(x)$ we see that $\varphi$ is well defined.

\begin{lemma}\label{lem-cont}
For every $\eps\in(0, 1)$, the function $\varphi$ defined above is continuous.
\end{lemma}

\begin{proof}
Consider a point $(x, p)\in\Rd\times\Rd$. Let $\varphi(x, p)=u\in\M(x)$. We claim that if $x_i^-=0$ for some $i\in\{1, \ldots, d\}$, then
\begin{equation}\label{12}
u_{ki}=0, \quad \text{for all} \quad k\neq i.
\end{equation}
Suppose \eqref{12} is not true and $u_{ki}>0$ for some $k\neq i$. We define $\bar{u}\in\M$ as follows,
\[\bar{u}_{lm}=\left\{\begin{array}{lll}
u_{j_1j_2} & \text{if}\; \; j_1\neq k, \; \text{or}\; j_2\neq i,
\\[5pt]
0 & \text{if}\;\; j_1=k,\; \text{and}\;\; j_2=i.
\end{array}
\right.
\]
It is easy to check that $\bar{u}\in\M(x)$. We also have
$$b(x, u)\cdot p + \Tilde{r}_\eps(x, u) > b(x, \bar{u})\cdot p + \Tilde{r}_\eps(x, \bar{u}).$$
which contradicts the fact that $u\in\Argmin_{u\in\M(x)} \{b(x, u)\cdot p + \Tilde{r}_\eps(x, u)\}$ and this proves \eqref{12}.
Let $(x_n, p_n)\to (x, p)$, as $n\to\infty$. We show that $\varphi(x_n, p_n)\to \varphi(x, p)$, as $n\to \infty$. Let
\begin{align*}
u_n &\;\df\; \varphi(x_n, p_n)\in\, \Argmin_{u\in\M(x_n)}\; \bigl\{b(x_n, u)\cdot p_n + \Tilde{r}_\eps(x_n, u)\bigr\},
\\
u &\;\df\; \varphi(x, p)\in\, \Argmin_{u\in\M(x)}\; \bigl\{b(x, u)\cdot p + \Tilde{r}_\eps(x, u)\bigr\}.
\end{align*}
Since the sequence $\{u_n\}$ is bounded we can assume that $u_n\to m_0\in\M$. It is also easy to see that $m_0\in\M(x)$. We need to show that
$m_0=u$.
Given $u$ as above we find $m_n\in\M(x_n)$ such that $m_n\to u$, as $n\to\infty$. Construction of $m_n$ is done in following two cases.

\noindent{Case 1:} Let  $u=0$. Therefore we let
$m_n=0\in\M(x_n)$ for all $n$.

\noindent{Case 2:} Let $u\neq 0$ . Then using \eqref{12} we see that whenever $u_{ki}>0$ for
some $k$ and $i\neq k$ we have $x_i^->0$. Let $\I(x)=\{i\; :\; u_{ki}>0\; \text{for some}\; k\neq i\}$.
Therefore $(x_n)_{i}^- >0$ for all $i\in\I(x)$ and large $n$. Define for large $n$,
$$\del^n_{ki}\:\df\: \frac{2\abs{x_n-x}}{(x_n)_{i}^-}, \quad \text{whenever}\; u_{ki}>0.$$
We set $\del^{n}_{ki}=0$ otherwise. Defne $(m_n)_{ki}=u_{ki}-\del^n_{ki}$ for all $k\neq i$.
Since $\del^n\to 0$ as $n\to\infty$, we have $m_n\in\M$ for all large $n$. Now we show that $m_n\in\M(x_n)$ for all large $n$.
To do this we note that for any $i\in\{1, \ldots, d\}$,
\begin{align*}
\sum_{j:j\neq i}(m_n)_{ij}(x_n)^-_j  = \sum_{j: u_{ij}>0}(u_{ij}-\del^n_{ij}).(x_n)^-_j
= \sum_{j: u_{ij}>0}u_{ij}.(x_n)^-_j- \abs{x_n-x}\sum_{j: u_{ij}>0} 2.
\end{align*}
If the set $\{j: u_{ij}>0\}$ is empty then the rhs of the above display is less than $(x_n)^+_i$. Otherwise we get
\begin{align*}
\sum_{j:j\neq i}(m_n)_{ij}(x_n)^-_j  &\leq \sum_{j: u_{ij}>0} u_{ij}x^-_j + \abs{x-x_n}\sum_{j: u_{ij}>0}(u_{ij}-2)
\\
&\leq x^+_i - \abs{x-x_n}\sum_{j: u_{ij}>0}1
\\
&\leq (x_n)^+_i + \abs{x-x_n}\Bigl(1-\sum_{j: u_{ij}>0}1\Bigr)
\\
&\leq (x_n)^+_i,
\end{align*}
where in the second inequality we use the fact that $u\in\M(x)$ and $u_{ij}\leq 1$. This proves that $m_n\in\M(x_n)$ for all large $n$. 
Hence using the
definition we have
$$b(x_n, u_n)\cdot p_n + \tilde{r}_\eps(x_n, u_n)\; \leq \; b(x_n, m_n)\cdot p_n + \tilde{r}_\eps(x_n, m_n),$$
and letting $n\to\infty$, we get
$$b(x_n, m_0)\cdot p + \tilde{r}_\eps(x, m_0)\; \leq \; b(x, u)\cdot p + \tilde{r}_\eps(x, u).$$
Therefore $m_0\in\M(x)$ is also a minimizer in $\M(x)$. By uniqueness property of the minimizer 
in $\M(x)$ we get $m_0=u$. Hence the proof.
\end{proof}

Let $\Lyap\in\calC^2(\Rd)$ be such that $\Lyap(x) =  \abs{x}^{k}$ for $\abs{x}\geq 1$
where $k\geq 1$. In fact, we can take $\Lyap$ to be non-negative. Also define
$$ h(x) \; \df\; |x|^k, \quad \text{for}\quad x\in\Rd.$$
Following lemma establishes uniform stability of the our system.

\begin{lemma}\label{lem-stability}
There exists a positive constants $c_5, c_6$, depending on $k$, such that
\begin{equation}\label{16}
\begin{split}
\sup_{u\in\M(x)}\, L^u \Lyap(x) &\leq c_5 - c_6\, h(x), \quad \text{for all}\quad x\in\Rd.
\end{split}
\end{equation}
\end{lemma}

\begin{proof}
Recall from \eqref{dc6} that for $u\in\M(x)$,
\begin{align*}
b_i(x, u) &= \ell_i + \mu_i x_i^- -\sum_{j:j\neq i}\mu_{ij}u_{ij}x_j^- -\gam_iq_i(x, u),
\\[5pt]
q_i(x, u) & = x_i^+ - \sum_{j:j\neq i}u_{ij}x^-_j.
\end{align*}
For $x\in\Rd$, and $u\in\M(x)$, we have $q_i(x, u)\geq 0$, and
\begin{align*}
\sum_{i=1}^d\abs{x_i} &= \sum_{i=1}^d(x^+_i + x^-_i)
\\
& \leq \sum_{i=1}^d\Bigl(q_i(x, u) +\sum_{j:j\neq i} u_{ij}x^-_j+ x^-_i\Bigr)
\\
& \leq \sum_{i=1}^d\Bigl(q_i(x, u) + d\, x^-_i\Bigr),
\end{align*}
where in the last line above we use the fact that $u_{ij}\leq 1$. Hence
\begin{equation}\label{17}
\abs{x}^2 \leq d\, \sum_{i=1}^d\Bigl(q_i(x, u) + d\, \abs{x^-_i}\Bigr)^2
\leq\; 2d \sum_{i=1}^d\Bigl((q_i(x, u))^2 + d^2\, \abs{x^-_i}^2\Bigr).
\end{equation}
We note that for any $u\in\M(x)$, we have
$$q_i(x, u)\cdot x_i= q_i(x, u)\cdot x^+_i\geq q^2_i(x, u).$$
Thus for $u\in\M(x)$ and $\abs{x}\geq 1$, we obtain
\begin{align}\label{18}
\sum_{i=1}^d b_i(x, u)\cdot\partial_i\Lyap(x)& \leq k\,\abs{x}^{k-2}\Bigl(\sum_{i=1}^d \abs{\ell_i} \abs{x_i}  - \mu_i\abs{x_i^-}^2
-x_i\sum_{j:j\neq i}\mu_{ij}u_{ij}x_j^- \nonumber
\\
&\, \hspace{.2in}-\gam_i q_i(x, u)x_i\Bigr)\nonumber
\\
&\leq k\,\abs{x}^{k-2}\Bigl(\sum_{i=1}^d \abs{\ell_i} \abs{x_i}  - \mu_i\,\abs{x_i^-}^2
-\gam_i\, q^2_i(x, u)\Bigr)\nonumber
\\
&\leq \kap_1\abs{x}^{k-1} -\kap_2\abs{x}^{k},
\end{align}
for some constants $\kap_1, \kap_2>0,$ where in the second inequality we use the fact that 
$$ x_i\sum_{j:j\neq i}\mu_{ij}u_{ij}x_j^-=x_i^+\, \sum_{j:j\neq i}\mu_{ij}u_{ij}x_j^-\geq 0,$$
for $u\in\M(x)$, and in the third inequality we use \eqref{17}. Now the proof can be seen using \eqref{18}.
\end{proof}

Next we discuss a convex analytic approach that assures the existence of an optimal control. To do this 
we introduce some more notations. For $U\in\Uadm$, we define
$$\varrho_{U}(x) \; \df\; \limsup_{T\to\infty}\, \frac{1}{T}\,\Exp^U_x[\int_0^T\Tilde{r}(X(t), U(t)) dt].$$
We use the notation $\Exp^U[\cdot]$ to express the dependency in $U$. For $\beta>0$, we define
\begin{align*}
\Uadm^\beta\;\df\; \{U\in\Uadm\; :\; \varrho_U(x)\leq \beta \; \text{for some }\, x\in\Rd\}.
\end{align*}
Let $\Usm^\beta=\Usm\cap\Uadm^\beta$. Denote
\begin{align*}
\varrho_{*} &\; \df \;\inf\{\beta > 0\, :\, \Uadm^\beta\neq\emptyset\},
\\[5pt]
\hat\varrho_{*} &\; \df \;\inf\{\beta > 0\, :\, \Usm^\beta\neq\emptyset\},
\\[5pt]
\tilde\varrho_{*} &\; \df \;\inf\{\uppi(\Tilde{r}) \, :\, \uppi\in\mathcal{G}\},
\end{align*}
where 
$$\uppi(\Tilde{r}) \;\df\; \int_{\Rd\times\M} \Tilde{r}(x, u) \uppi(\D{x}, \D{u}),$$
and,
\begin{align}\label{0001}
\mathcal{G}\;\df\;\Big\{\uppi\in\cP(\Rd\times\M)\; &:\; \int_{\Rd\times\M} L^u f(x) \uppi(\D{x}, \D{u}) =0, \; \text{for all}\,
f\in\calC_c^\infty(\Rd), \nonumber
\\
&\, \hspace{.2in}\, \text{and}\, \int_{\Rd\times\M}1_{\M(x)}(u)\uppi(\D{x}, \D{u})=1 \Big\}.
\end{align}
In what follows, we denote by $\uptau(A)$ the \emph{first exit time} of a process
$\{X_{t}\,,\;t\in\R_{+}\}$ from a set $A\subset\Rd$, defined by
\begin{equation*}
\uptau(A) \;\df\;\inf\;\{t>0\, : \,X_{t}\not\in A\}\,.
\end{equation*}
The open ball of radius $R$ in $\Rd$, centered at the origin,
is denoted by $\B_{R}$, and we let $\uptau_{R}\;\df\;\uptau(\B_{R})$,
and $\tc_{R}\df \uptau(\B^{c}_{R})$.

\begin{theorem}\label{T-exist}
We have 
\begin{enumerate}
\item[(a)]$\varrho_{*}=\hat\varrho_{*}=\Tilde\varrho_{*}$.
\item[(b)] There exists $v\in\Usm$ such that $\varrho_v=\varrho_{*}$.
\end{enumerate}
\end{theorem}

\begin{proof}
Let $U\in \Uadm$. Applying It\^{o}'s formula, it follows from \eqref{16} that
\begin{equation*}
\frac{1}{T}\Big(\Exp_x^U\big[\Lyap(X({\uptau_R\wedge T}))\big] - \Lyap(x)\Big)
\leq c_5 - c_6 \frac{1}{T}\,\Exp_x^U\Bigl[\int_0^{\uptau_R\wedge T} h(X(s))\D{s}\Bigr].
\end{equation*}
Letting $R\to\infty$, and then $T\to\infty$, we see that
$$\limsup_{T\to\infty}\, \frac{1}{T}\, \Exp_x^U\Bigl[\int_0^{ T} h(X(s))\D{s}\Bigr]<\infty.$$
Since $h$ is inf-compact, the above display implies that the \textit{mean-empirical measures} $\{\zeta_{x, t}\, : t> 0\}$,
defined by
$$\zeta_{x, t}(A\times B)\; \df\; \frac{1}{t}\Exp_x\Big[\int_0^t 1_{A\times B}(X(s), U(s))\, \D{s}\Big],$$
are tight. Since
$$ \int_{\Rd\times\M} 1_{M^c(x)}(u)\zeta_{x, t}(\D{x}, \D{u})\; =\; 0,$$
and $(x, u)\mapsto 1_{M^c(x)}(u)$ is lower-semicontinuous, it is easy to see that every sub sequential limit of 
$\{\zeta_{x, t}\, : t>0\}$, as $t\to\infty$, lies in $\mathcal{G}$. Also if $\uppi$ is one of the limits of $\{\zeta_{x, t}\, : t > 0\}$, we
get 
$$\uppi(\Tilde{r}) \;\leq \;  \limsup_{T\to\infty}\, \frac{1}{T}\, \Exp_x^U\Bigl[\int_0^{ T} 
\tilde{r}(X(s), U(s))\D{s}\Bigr].$$
This shows that $\tilde{\varrho}_{*}\leq \varrho_{*}$. Let $\uppi\in\mathcal{G}$. Using disintegration of measure we write
$\uppi(\D{x}, \D{u}) = v(\D{u} | x) \mu_v(\D{x})$. Therefore by definition, we have
$$\int_{\Rd} L^v f(x) \mu_v(\D{x})= \int_{\Rd} \int_{M}L^u f(x) v(\D{u}|x) \mu_v(\D{x}) =0, \quad \forall \; f\in\calC_c^\infty(\Rd).$$
Hence applying \cite[Theorem~2.6.16]{ari-bor-ghosh} we see that $\mu_v(\D{x})$ has locally strictly positive density. In
particular, $\mu_v(\D{x})$ is mutually absolutely continuous with respect to the Lebesgue measure on $\Rd$.
Combining these observations with the fact that 
$$\int_{\Rd\times\M}1_{M^c(x)}(u)\uppi(\D{x}, \D{u})=\int_{\Rd\times\M}v(\M^c(x) | x)\mu_v(\D{x})=0,$$
we conclude that $v(\M^c(x) | x)=0$ almost surely with respect to the Lebesgue measure on $\Rd$. 
Since $v(\M^c(x) | x)$ is Borel-measurable we can modify $v$ on a Borel set of measure $0$
so that $v(\M^c(x) | x)=0$ holds everywhere. Hence
the stationary solution $X(t)$ corresponds to the Markov control $v(\cdot | x)$ would satisfy
$v(\M(X(t)) | X(t))=1$ almost surely. Thus $v(\cdot | x)$ is an admissible Markov control.
By ergodic theorem \cite{yosida},\cite[Theorem~1.5.18]{ari-bor-ghosh} it is know that
$$\limsup_{T\to\infty}\, \frac{1}{T}\Exp^v_x\Bigl[\int_0^T \Tilde{r}(X(s), v(X(s))) \D{s}\Bigr]\xrightarrow[n\to\infty]{}
\uppi(\Tilde{r}), \quad\text{for almost every }\, x\in\Rd.$$
Note that if $\Tilde{r}(X(s), v(X(s)))$ is continuous we could use weak convergence of mean-empirical
measures, corresponding to $X$, to justify the above limit. But  $v$ need not be continuous, in general. So we use ergodic
theorem to pass the limit.
Thus $\hat\varrho_{*}\leq \tilde\varrho_{*}$. But by definition $\varrho_*\leq \hat{\varrho}_{*}$. Thus we have
$\varrho_*= \hat{\varrho}_{*}=\tilde\varrho_{*}$. This proves (a).

To prove (b), we consider a sequence $\uppi_n\in \mathcal{G}$ along which the infimum is achieved. Applying Lemma~\ref{lem-stability} we obtain that the measures $\{\uppi_n\}$ are tight and
 thus it has a convergent subsequence.  Let $\uppi$ be one
of the subsequential limits and $\uppi_n\to\uppi$, after relabelling, as $n\to\infty$. Lower-semicontinuity of $(x, u)\mapsto 1_{\M^c(x)}(u)$ implies
that $\uppi\in\mathcal{G}$. Moreover,
$$\uppi(\Tilde{r})\leq \liminf_{n\to\infty}\,\uppi_n(\Tilde{r})=\varrho_{*}.$$
Thus the infimum is achieved at $\uppi$ and the Markov control is obtained by disintegrating 
$\uppi(\D{x}, \D{u})=v(\D{u} | x) \mu_v(\D{x})$ where $v(\D{u}| \cdot)$ is our required Markov control.
\end{proof}

\begin{remark}
We observe that the arguments of Theorem~\ref{T-exist} continue to hold if we replace $\tilde{r}$ by $\tilde{r}_\eps$ and
$\varrho$ by $\varrho_\eps$. The above theorem also establishes independence of $\varrho$ (and $\varrho_\eps$) from
$x$.
\end{remark}
The above existence result of optimal Markov control is purely analytic and it does not provide any characterization of the optimal controls.
To prove Theorem~\ref{T-optimality} we need to find an optimal control with some regularity properties such as continuity.
But the above result does not say anything about the regularity properties of optimal Markov controls. Therefore we will
analyze the associated HJB to extract more information about the optimal Markov controls.

For $\al\in(0, 1)$, we define
\begin{equation}\label{20}
V^\al_\eps(x)\; \df\; \inf_{U\in\Uadm}\Exp_x\Big[\int_0^\infty e^{-\al\, s} \Tilde{r}_\eps(X(s), U(s))\, \D{s}\Big].
\end{equation}
The following result characterize the $\alpha$-discounted problem.

\begin{theorem}\label{T-discount}
For every $\eps\in(0, 1)$, $V^\al_\eps\in\calC^2(\Rd)$  satisfies the HJB
\begin{equation}\label{40}
\min_{u\in\M(x)} \Bigl(L^u V^\al_\eps(x) + \Tilde{r}_\eps(x, u)\Bigr) -\al V^\al_\eps= 0.
\end{equation} 
Moreover, the unique minimizing selector is an optimal Markov control for \eqref{20}.
\end{theorem}

Before we prove the theorem we need to establish some regularity properties of the Hamiltonian. To do so we define,
$$ H_\eps(x, p)\; \df \; \min_{u\in\M(x)}\{p\cdot b(x, u) + \Tilde{r}_\eps(x, u)\}.$$

\begin{lemma}\label{lem-holder}
Let $\calK\subset\Rd$ be compact. Then for any $R>0$, there exists a constant $c=c(\calK, R, \eps)$, 
depending on $\calK, R$ and $\eps$, such that
$$\Big|H_\eps(x, p)- H_\eps(y, q) \Big|\; \leq \; c\; (\abs{x-y}^{\frac{1}{4}} + \abs{p-q}), \quad \text{for all}\quad x,\, y\in\calK,
\; \text{and}, \; p, \, q\in\B_{R}(0).$$ 
\end{lemma}

\begin{proof}
We note that for any $x\in \calK$, we have
$$ \Big|H_\eps(x, p)- H_\eps(x, q) \Big|\; \leq \; \sup_{(x, u)\in\calK\times\M}\abs{b(x, u)} \; \abs{p-q}.$$
Therefore it is enough to show that for any $x, y\in\calK$, $q\in\B_R(0)$,
\begin{equation}\label{21}
H_\eps(x, q)- H_\eps(y, q) \leq c(\calK, R, \eps)\;   \abs{x-y}^{\frac{1}{4}}.
\end{equation}
Since $b$ and $r$ are locally Lipschitz,
we have a constant $\kap>0$, depending on $\calK$, such that
\begin{equation}\label{22}
\abs{b(x, u)-b(x, \bar{u})} + \abs{\Tilde{r}(x, u)- \Tilde{r}(x, \bar{u})} \; \leq \; \kap 
\sum_{i\neq j}\abs{u_{ij}x_j^- -\bar{u}_{ij}x_j^-}, \quad
\text{for}\;\; u, \bar{u}\in\M(x), \; x\in\calK.
\end{equation}
 Let $u\in\M(y)$ be a minimizer of 
$$\bigl\{q\cdot b(y, u') + \Tilde{r}_\eps(y, u')\bigr\},$$
in $\M(y)$. In fact, this is the unique minimizer because of the strict convexity of the functional.
Define $\theta\;\df\; [\frac{2\kap}{\eps}(\abs{q} + 1)]^{1/2}\vee d$. We claim that 
\begin{equation}\label{23}
\text{for any $\del\in(0, \frac{1}{\theta})$, there is no $u_{ij}> \theta\,\del$, if $y_j^-< \del^2.$}
\end{equation}
The above implies that either $u_{ij}\leq \theta\del$ or $\sum_{j: u_{ij}> \theta\del} y_j^-\geq \del^2\sum_{j: u_{ij}> \theta\del} 1$.
Supposing the contrary, we assume that \eqref{23} is not ture i.e., there exists $i, j, i\neq j$, such that 
\begin{center}
$u_{ij}\in (\del\theta, 1]$,\; and\; $y_j^- <\del^2$.
\end{center}
We define $\tilde{u}\in\M(y)$ as follows
\[
\tilde{u}_{j_1j_2}=\left\{\begin{array}{lll}
u_{j_1j_2} & \text{if}\; j_1\neq i, \text{or}\; j_2\neq j,
\\[5pt]
0 & \text{otherwise}.
\end{array}
\right.
\]
Then using \eqref{22} we obtain
\begin{align*}
\{q\cdot b(y, u) + \Tilde{r}_\eps(y, u)\}-\{q\cdot b(y, \tilde{u}) + \Tilde{r}_\eps(y, \tilde{u})\}
&\geq -\kap\, \abs{q}\del^2 - \kap \del^2 + \eps \abs{u_{ij}}^2
\\[5pt]
&\geq -(\abs{q}+ 1)\,\kap\del^2 + \eps\theta^2\del^2
\\[5pt]
&> 0.
\end{align*}
This contradicts that fact that $u\in\M(y)$ is a minimizer. This proves the claim \eqref{23}.

Now we proceed to prove \eqref{21}. We note that it is enough to show \eqref{21} for $\abs{x-y}<\frac{1}{\theta}$. Let $\abs{x-y}<\frac{1}{\theta}$. 
We define
 $\cA(y)\subset\{1, \ldots, d\}$ as follows.
\begin{align*}
\cA(y) &\; \df\; \{i\in\{1, \ldots, d\} \; :\; \sum_{j: u_{ij}> \theta \abs{x-y}^{1/4}} y_j^-\geq \abs{x-y}^{1/2}\}.
\end{align*}
Note that $\cA(y)$ could be empty. Observe that for any $i\in\cA^c$ we either have $u_{ij}\leq \theta \abs{x-y}^{1/4}$ for all
$j\neq i$, or $y_j^-\leq \abs{x-y}^{1/2}$ for all $j$ satisfying $u_{ij}>\theta\abs{x-y}^{1/4}$. But the second
situation does not occur due to \eqref{23}. Thus for $i\in\cA^c(y)$ we have $u_{ij}\leq \theta \abs{x-y}^{1/4}$
for all $j\neq i$.
Therefore there exists
a constant $\kap_1$, depending on $\calK$ and $\theta$, such that for all $i\in\cA^c(y)$,
\begin{equation}\label{24}
\sum_{j:j\neq i} u_{ij}y_j^-\leq \kap_1 \abs{x-y}^{\frac{1}{4}}.
\end{equation}
Now we define $\bar{u}\in\M$ as follows
\[
\bar{u}_{ij}=\left\{\begin{array}{lll}
0, & \text{it}\; i\in\cA^c(y),
\\[5pt]
0, & \text{if}\; i\in\cA(y)\;\; \text{and} \; u_{ij}\leq \theta\abs{x-y}^{1/4},
\\[5pt]
(u_{ij}-\theta\abs{x-y}^{1/4}), & \text{otherwise}.
\end{array}
\right.
\]
We show that $\bar{u}\in \M(x)$. To do so we only need to check that for $i\in\cA(y)$,
\begin{equation*}
\sum_{j: j\neq i} \bar{u}_{ij}x_j^- \;\leq\;  x^+_i.
\end{equation*}
For $i\in\cA(y)$, 
\begin{align*}
\sum_{j: j\neq i} \bar{u}_{ij}\, x_j^- &\;=\; \sum_{j: u_{ij}> \theta\abs{x-y}^4} ({u}_{ij}-\theta\abs{x-y}^{1/4})\, x_j^-
\\ 
&\leq \sum_{j: u_{ij}> \theta\abs{x-y}^4} ({u}_{ij}-\theta\abs{x-y}^{1/4})\,( y_j^- + \abs{x-y})
\\[5pt]
&\leq \; \sum_{j: j\neq i} u_{ij} y_j^- + \sum_{j: u_{ij}> \theta\abs{x-y}^4}(u_{ij}\abs{x-y}-\theta \abs{x-y}^{1/4}y_j^-)
\\
&\leq y_i^+ + \sum_{j: u_{ij}> \theta\abs{x-y}^4}(u_{ij}\abs{x-y}-\theta \abs{x-y}^{1/4}y_j^-)
\\
&\leq x^+_i + \abs{x-y} (1+\sum_{j:j\neq i}u_{ij}) - \theta \abs{x-y}^{3/4}
\\
&\leq x^+_i,
\end{align*}
where in the last line we use the fact that $\abs{x-y}^{1/4}d\leq \frac{d}{\theta^{1/4}}<\theta$. This shows $\bar{u}\in\M(x)$.
Also by the construction of $\bar{u}$ and \eqref{24} it is evident that for all $i, j\in\{1, \ldots, d\}, i\neq j,$
$$ \abs{u_{ij}y_j^- -\bar{u}_{ij}x_j^-}\leq \kap_2 \abs{x-y}^{1/4}, \quad \text{and} 
\; \;\Big|\abs{u_{ij}}^2-\abs{\bar{u}_{ij}}^2\Big|\leq \kap_2\abs{x-y}^{1/4},$$
for some constant $\kap_2$. Hence \eqref{21} follows using the above display and \eqref{22}.
\end{proof}

Now we are ready to prove Theorem~\ref{T-discount}.
\begin{proof}[Proof of Theorem~\ref{T-discount}]
The proof is in spirit of \cite[Theorem~3.5.6]{ari-bor-ghosh}. Two key ingredients for the proof are Lemma~\ref{lem-cont}
and Lemma~\ref{lem-holder}. Recall that $\B_R$ denotes the ball of radius $R$ around $0$. It is known
that there exists a solution $\varphi_R\in\calC^{2, \beta}(\B_R), \beta\in(0, 1/4),$ satisfying
\begin{equation*}
\begin{split}
\frac{1}{2}\sum_{i, j=1}^d a_{ij}\partial_{ij}\varphi_R(x) + H_\eps(x, \grad\varphi(x)) &\;=\; \al\varphi_R(x),\quad x\in\B_R,
\\[5pt]
\varphi &=0, \quad \text{on}\;\; x\in\partial\B_R.
\end{split}
\end{equation*}
The existence result follows from \cite[Theorem~11.4]{gil-tru} and maximum-principle together with Lemma~\ref{lem-cont}
and Lemma~\ref{lem-holder}. See \cite[Theorem~3.5.3]{ari-bor-ghosh} for a similar argument. Following
a similar argument as in \cite[Theorem~3.5.6]{ari-bor-ghosh} together with Lemma~\ref{lem-cont} we can show that
\begin{equation}\label{25}
\varphi_{R}(x) = \inf_{U\in\Uadm}\Exp_x\Big[\int_0^{\uptau_R}e^{-\al s}\Tilde{r}_\eps(X(s), U(s))\D{s}\Big].
\end{equation}
Therefore choosing $k >m$ in Lemma~\ref{lem-stability} and using \eqref{cost1} we see that
\begin{equation}\label{33}
V^\al_\eps(x)\leq \inf_{U\in\Uadm}\Exp_x[\int_0^\infty e^{-\al s}h (X(s))\, \D{s}]<\frac{\kap_1}{\al} + \Lyap(x),
\end{equation}
for some constant $\kap_1$. Hence using \eqref{25} we get that $\varphi_R(x)<\frac{\kap_1}{\al} + \Lyap(x)$
for all $x$. This shows that $\varphi_R$ is locally bounded, uniformly in $R$. Hence using standard theory of elliptic PDE's and 
Lemma~\ref{lem-cont} we obtain that for any domain $D\subset\Rd$ and $p\geq 1$,
$$\norm{\varphi}_{\Sob^{2, p}(D)} \leq \kap_2,$$
where the constant $\kap_2$ is independent of $R$. Hence we can extract a subsequence of $\varphi_R$
that converges to $\varphi$ as $R\to\infty$ in $\Sobl^{2, p}(\Rd), p\geq 2d,$ and in $\calC^{1, r}_{\mathrm{loc}}(\Rd)$
for $r\in(0, 1/4)$. Using locally Lipschitz property of $H_\eps(x, \cdot)$, by Lemma~\ref{lem-holder}, we obtain that $\varphi$ is a weak solution
to
\begin{equation}\label{30}
\min_{u\in\M(x)}(L^u\varphi(x) + \Tilde{r}_\eps(x, u))= \al \varphi(x), \quad \text{for}\quad x\in\Rd.
\end{equation}
Lemma~\ref{lem-holder} and the theory of elliptic PDE's give us $\varphi\in\calC^2(\Rd)$. From \eqref{25} we also have 
$\varphi\leq V^\al_\eps$. To show that equality we consider the minimizing selector $u(\cdot)$ of \eqref{30}. Existence of such selector
is assured from Lemma~\ref{lem-cont}. Hence we have
$$\frac{1}{2}\sum_{i, j=1}^d a_{ij}\partial_{i, j}\varphi + b(x, u(x))\cdot\grad\varphi(x) + \Tilde{r}_\eps(x, u(x)) =\al\varphi(x).$$
Since $u(x)\in\M(x)$ for all $x$, the solution $X$ corresponding to this Markov control is admissible. Therefore
applying It\^{o}'s formula we get
$$\Exp_x[e^{-\al t\wedge \uptau_R}\varphi(X(t\wedge\uptau_R))]-\varphi(x) = \Exp^u_x\Big[\int_0^{t\wedge \uptau_R}
e^{-\alpha s}\, \Tilde{r}_\eps(X(s), u(X(s)))\, \D{s}\Big].$$
Now we use the non-negativity of $\varphi$ to conclude that 
$$\Exp^u_x\Big[\int_0^{t\wedge \uptau_R}
e^{-\alpha s}\,\Tilde{r}_\eps(X(s), u(X(s)))\, \D{s}\Big]\leq \varphi(x),$$
Letting $R\to\infty$, in the above display and using Fatou's lemma we get
$$\Exp^u_x\Big[\int_0^{t}
e^{-\alpha s}\,\Tilde{r}_\eps(X(s), u(X(s)))\, \D{s}\Big]\leq \varphi(x).$$
This shows $V^\al_\eps=\varphi$ and $u$ is an optimal Markov control.
\end{proof}

By $\tilde{\Uadm}_{\mathrm{SM}}$ we denote the set of all admissible deterministic Markov controls, i.e., 
$\tilde{\Uadm}_{\mathrm{SM}}$ denotes collection of all measurable $v:\Rd\to\M$ such that $v(x)\in\M(x)$ for  all $x$. We recall that $\tc_r$ denotes the hitting time to the ball $\B_r$, i.e.,
$$\tc_r\; \df\; \inf\{t\geq 0 \; :\; X(t)\in \B_r\}.$$

\begin{lemma}\label{L4.5}
Let $\bar{V}^\al_\eps\;\df\; V^\al_\eps-V^\al_\eps(0)$. Then $\bar{V}^\al_\eps$ is bounded in $\Sobl^{2, p}(\Rd), p\geq 1,$
and $\{\al V^\al_\eps(0)\}_{\al\in(0, 1)}$ is also bounded. Let $(V_\eps, \varrho)\in\Sobl^{2, p}(\Rd)\times\R$ be any sub-sequnetial
limit of $(\bar{V}^\al_\eps, \al V^\al_\eps(0))$, as $\al\to 0$, then we have $V_\eps\in\calC^2(\Rd)$ that satisfies,
\begin{equation}\label{35}
\min_{u\in\M(x)}(L^u V_\eps(x) + \tilde{r}_\eps(x, u)) =\varrho.
\end{equation}
Moreover $\varrho\leq \varrho_{\eps}$. Furthermore,
\begin{enumerate}
\item[(a)] $V_\eps(x) \; \leq \; \liminf_{r\downarrow 0}\, \inf_{v\in\tilde{\Uadm}_{\mathrm{SM}}}\,
\Exp^v_x\Bigl[\int_0^{\tc_r}\bigl(\Tilde{r}_\eps(X(s), v(X(s))) -\varrho\bigr)\D{s}\Bigr].$
\item[(b)] If $u_\eps\in \tilde{\Uadm}_{\mathrm{SM}}$ denote the minimizing selector of \eqref{35} then 
$$ V_\eps(x) \;\geq \; -\varrho \Exp_x^{u_\eps}[\tc_r] - \sup_{\B_r} V_\eps, \quad \text{for all}\,\; x\in\B^c_r.$$
\end{enumerate}
\end{lemma}

\begin{proof}
From \eqref{33} we obtain that for any $R>0$, 
$$ \al\max_{x\in\B_R}V^\al_\eps(x) <\kap_R,$$
where the constant $\kap_R$ does not depend on $\al$ and $\eps$. Hence applying 
\cite[Lemma~3.6.3]{ari-bor-ghosh} (see also \cite[Lemma~3.5]{ari-bis-pang})
we get that $\bar{V}^\al_\eps$ is bounded in $\Sobl^{2, p}(\Rd), \, p\geq 1$. Also  
boundedness $\{\al V^\al_\eps(0)\}_{\al\in(0, 1)}$ follows from the above display. For $p>2d$, we see that
any sub-sequential limit also converges in $\calC^{1, r}_{\mathrm{loc}}(\Rd)$. Hence any sub-sequential limit $(V_\eps, \varrho)$ would satisfy
\eqref{35}. We can improve regularity of $V_\eps$ to $\calC^2$ using Lemma~\ref{lem-holder}. Now we show that
$\varrho\leq \varrho_\eps$. Let $U\in\Uadm$ be any admissible control. We also assume that
$$\limsup_{T\to\infty} \frac{1}{T}\Exp^U_x\Bigl[\int_0^T \Tilde{r}_\eps(X(s), U(s))\D{s}\Bigr]<\infty,$$
for some $x$. It is easy to see that
$$\limsup_{T\to\infty} \frac{1}{T}\Exp^U_x\Bigl[\int_0^T \Tilde{r}_\eps(X(s), U(s))\D{s}\Bigr]
=\limsup_{N\to\infty} \frac{1}{N}\Exp^U_x\Bigl[\int_0^N \Tilde{r}_\eps(X(s), U(s))\D{s}\Bigr],$$
where $N$ runs over natural numbers. In fact, it is easy to see that for RHS of above display is smaller that LHS, and for any $T_n\to\infty$, we have
\begin{align*}
\limsup_{T_n\to\infty} \frac{1}{T_n}\Exp^U_x\Bigl[\int_0^{T_n} \Tilde{r}_\eps(X(s), U(s))\D{s}\Bigr]
&=\limsup_{T_n\to\infty} \frac{\lfloor T_n+1\rfloor}{\lfloor T_n\rfloor}\, \frac{1}{\lfloor T_n+1\rfloor}
\Exp^U_x\Bigl[\int_0^{\lfloor T_n+1\rfloor} \Tilde{r}_\eps(X(s), U(s))\D{s}\Bigr]
\\
&\leq \limsup_{N\to\infty} \frac{1}{N}\Exp^U_x\Bigl[\int_0^N \Tilde{r}_\eps(X(s), U(s))\D{s}\Bigr].
\end{align*}
Define
$$a_N \; \df\; \Exp^U_x\Bigl[\int_{N-1}^N \Tilde{r}_\eps(X(s), U(s))\D{s}\Bigr], \quad N\geq 1.$$
For $\beta=e^{-\al}$, we have
\begin{align*}
\sum_{N=1}^\infty\beta^N a_N &\geq e^{-\al} \sum_{N=1}^\infty\Exp_x^U\Bigl[\int_{N-1}^Ne^{-\al\, s} \Tilde{r}_\eps(X(s), U(s))\D{s}\Bigr]
\\[5pt]
&=e^{-\al}\Exp_x^U\Bigl[\int_0^\infty e^{-\al\, s} \Tilde{r}_\eps(X(s), U(s))\D{s}\Bigr].
\end{align*}
Therefore applying \cite[Theorem~2.2]{Sz-Fil} we get
\begin{align}\label{36}
\limsup_{T\to\infty} \frac{1}{T}\Exp^U_x\Bigl[\int_0^T \Tilde{r}_\eps(X(s), U(s))\D{s}\Bigr]
&\geq \limsup_{\al\to 0} (1-\beta)\sum_{N=1}^\infty\beta^N a_N\nonumber
\\[5pt]
& \geq \limsup_{\al\to 0} (1-e^{-\al})e^{-\al} V^\eps_\al(x).
\end{align}
Since $\lim_{\al\to 0}\frac{1-e^{-\al}}{\al}=1$ we obtain from \eqref{36} that
$$\limsup_{T\to\infty} \frac{1}{T}\Exp^U_x\Bigl[\int_0^T \Tilde{r}_\eps(X(s), U(s))\D{s}\Bigr]\geq \varrho,$$
where we use the fact that for any $x\in\Rd$, $\al(V^\al_\eps(x)-V^\al_\eps(0))\to 0$ as $\al\to 0$. Since $U\in\Uadm$
is arbitrary we have $\varrho_{\eps}\geq \varrho$. (a) follows from the same argument as in \cite[Lemma~3.7.8]{ari-bor-ghosh}.
In fact, following \cite[Lemma~3.7.8]{ari-bor-ghosh} one obtains that for any $r>0$,
\begin{equation}\label{42}
V_\eps(x) \leq \; \inf_{v\in\tilde{\Uadm}_{\mathrm{SM}}}\Exp_x^v\Big[\int_0^{\tc_r}(\Tilde{r}_\eps(X(s), v(X(s)))-\varrho)\D{s}
+ V_\eps(X(\tc_r))\Big].
\end{equation}
Let $u^\al_\eps$ be the minimizing selector of \eqref{40}. Then applying Lemma~\ref{lem-cont} we get that
$u^\al_\eps\to u_\eps$ pointwise, as $\al\to 0$. Also by Theorem~\ref{T-discount} we have
\begin{align*}
\bar{V}^\al_\eps(x) & \;= \Exp_x^{u^\al_\eps}\Big[\int_0^{\tc_r} e^{-\al s}\bigl(\tilde{r}_\eps(X(s), u^\al_\eps(X(s))) -\varrho\bigr)\D{s}\Big]
+ \Exp_x^{u^\al_\eps}\Big[ \bar{V}^\al_\eps(X(\tc_r))\Big]
\\
&\, \quad \hspace{.2in} + \Exp_x^{u^\al_\eps}\Big[ \al^{-1} (1-e^{-\al\tc_r}) (\varrho-\al V^\al_\eps(X(\tc_r))\Big],
\quad \forall\, x\in\B_r^c.
\end{align*}
Since $u^\al_\eps\to u_\eps$, using Lemma~\ref{lem-stability} we have (see also \cite[Lemma~3.8]{ari-bis-pang})
$$\Exp_x^{u^\al_\eps}[\tc_r]\xrightarrow[\al\to 0]{}\Exp_x^{u_\eps}[\tc_r], \quad \text{for all}\, \, x\in\B^c_r.$$
Since $(1-e^{-\al s})\leq \al s,$ for $s\geq 0$, combining above two display we get, as $\al\to 0$, that
$$ V_\eps(x) \geq -\varrho \Exp_x^{u_\eps}[\tc_r]- \sup_{y\in\B_r} V_\eps(y), \quad \text{for}\; x\in\B_r^c.$$
\end{proof}

The following lemma establishes optimality of $u_\eps$ that we choose above.
\begin{lemma}\label{L4.6}
Every sub-sequential limit $V_\eps$ that we obtain in Lemma~\ref{L4.5} is in $\sorder(\abs{x}^k)$ for $k>m,$ where
$m$ is given by \eqref{cost1}. Also if $u_\eps$ is the minimizing selector in \eqref{35} we have
\begin{equation}\label{43}
\varrho_{\eps}=\varrho=\limsup_{T\to\infty}\, \frac{1}{T}\Exp_x\Bigl[\int_0^T\Tilde{r}_\eps(X(s), u_\eps(X(s)))\D{s}\Bigr],
\quad \forall\; x\in\Rd.
\end{equation}
\end{lemma}

\begin{proof}
From Lemma~\ref{L4.5} (see \eqref{42}) we obtain  that
\begin{equation}\label{41}
 \abs{V_\eps(x)}\; \leq \; \sup_{v\in\tilde{\Uadm}_{\mathrm{SM}}}\Exp_x^v\Bigl[\int_0^{\tc_r}
\big(\tilde{r}_\eps(X(s), v(X(s))) + \varrho_{*}\big)\D{s}\Bigr] + \sup_{B_r} V_\eps, \quad \forall\; x\in\B_r^c.
\end{equation}
 For $v\in\Tilde{\Uadm}_{\rm SM}$, we define
$$L^{v}f(x) \;=\; \frac{1}{2}\sum_{i, j=1}^d a_{ij}\partial_{ij} f(x) + b(x, v(x))\cdot\grad f(x), \quad f\in\calC^2(\Rd).$$
Considering $k=m$ in Lemma~\ref{lem-stability} and applying Dynkin's formula we obtain, for $v\in\tilde{\Uadm}_{\mathrm{SM}}$,
that
\begin{align*}
\Exp^v_x[\Lyap(X(t\wedge\tc_r))]-\Lyap(x) & = \Exp^v_x\Big[ \int_0^{t\wedge\tc_r} L^v\Lyap(X(s))\D{s}\Big]
\\[5pt]
&= \Exp^v_x\Big[ \int_0^{t\wedge\tc_r}1_{\{v(X(s))\in\M(X(s))\}} L^v\Lyap(X(s))\D{s}\Big]
\\[5pt]
&\leq  \Exp^v_x\Big[ \int_0^{t\wedge\tc_r}\Big(c_5- c_6\abs{X(s)}^m\Big)\D{s}\Big].
\end{align*}
Now $\Lyap$ being non-negative, letting $t\to\infty$, in the above display we obtain
$$\Exp^v_x\Big[ \int_0^{\tc_r}\Bigl(\abs{X(s)}^m-\frac{c_5}{c_6}\Bigr)\D{s}\Big]\leq \frac{1}{c_6}\Lyap(x), \quad \text{for}\; \; x\in\B^c_r.$$
Choose $r\geq$ large enough so that $\abs{x}\geq \frac{1}{2}\vee\frac{4c_5}{c_6}$ for $x\in\B_r^c$.
Therefore
$$\Exp^v_x\Big[ \int_0^{\tc_r}\frac{1}{4}(1+\abs{X(s)}^m)\, \D{s}\Big]\leq
\Exp^v_x\Big[ \int_0^{\tc_r}\frac{1}{2}\abs{X(s)}^m\, \D{s}\Big]\leq \frac{1}{c_6}\Lyap(x), \quad \text{for}\; \; x\in\B^c_r.$$
Thus \eqref{cost1} and \eqref{41} gives that $V_\eps\in\order(\abs{x}^m)$. Hence $V_\eps\in\sorder(\abs{x}^k), k>m$.

Now let $u_\eps$ be the minimizing selector of \eqref{35}. We observe that $u_\eps\in\tilde{\Uadm}_{\rm SM}$. Moreover, 
using Lemma~\ref{lem-stability} we see that $u_\eps$ is stable with 
$$\limsup_{T\to\infty}\, \frac{1}{T}\Exp^{u_\eps}_x\Big[\int_0^T\abs{X(t)}^k\, \D{t}\Big]\, <\; \infty,$$
for any $k\geq 1$. Thus if $\mu_{u_\eps}$ denote the invariant measure corresponding to the Markov control $u_\eps$
we have 
$$\int_{\Rd} \abs{x}^k \mu_{u_\eps}(\D{x}) \; <\; \infty, \quad \text{for}\; k\geq 1.$$
Since $V_\eps\in\sorder(\abs{x}^k), k>m$, it follows from \cite[Proposition~2.6]{ichihara-sheu} 
that 
$$\lim_{T\to\infty}\frac{1}{T}\Exp^{u_\eps}_x\Big[\abs{V_\eps(X(T))}\Big]=0, \quad \text{for}\quad x\in\Rd.$$
Thus \eqref{43} follows by an application of Dynkin's formula to \eqref{35}.
\end{proof}

To this end we define 
$$ H(p, x)\; \df\; \inf_{u\in\M(x)}\{p\cdot b(x, u) + \Tilde{r}(x, u)\}.$$
the following result is similar to Lemma~\ref{lem-holder}.

\begin{lemma}\label{lem-holder-0}
Let $\calK\subset\Rd$ be compact. Then for any $R>0$, there exists a constant $c=c_{\calK, R}$, 
depending on $\calK, R$ , such that
$$\Big|H(x, p)- H(y, q) \Big|\; \leq \; c\, (\abs{x-y}^{\nicefrac{1}{2}} + \abs{p-q}), \quad \text{for all}\quad x,\, y\in\calK,
\; \text{and}, \; p, \, q\in\B_{R}(0).$$ 
\end{lemma}

\begin{proof}
We note that the proof of Lemma~\ref{lem-holder} is not applicable here to conclude the result. In fact, the constant $\theta$
defined in Lemma~\ref{lem-holder} tends to infinity as $\eps\to 0$. However, we adopt a similar technique to establish
H\"{o}lder regularity of the Hamiltonian. In view of Lemma~\ref{lem-holder}, we see that it is enough to show the
following: there
exists $\theta_1$, depending on $\calK, R$, such that for any $u\in\M(y)$ we could find $\bar{u}\in\M(x)$ such that
\begin{equation}\label{46}
|u_{ij}y_j^--\bar{u}_{ij}x_j^-|\leq \theta_1\, \abs{x-y}^{\nicefrac{1}{2}}, \quad \text{for all}\;\; i\neq j.
\end{equation}
and $x, y\in\calK$. Let $u\in\M(y)$. Let $\theta$ be a positive number that will be chosen later. Define
$$\cA(\theta)\;\df\; \Big\{i\in\{1, \ldots, d\}\; :\; \sum_{j: u_{ij}>\theta\abs{x-y}^{1/2}} y_j^-\leq \abs{x-y}^{1/2}\Big\}.$$
We define $\bar{u}\in\M$ as follows,
\[\bar{u}_{ij}=\left\{
\begin{array}{lll}
0, & \text{if} \; i\in\cA(\theta),
\\[5pt]
0, & \text{if} \; i\in\cA^c(\theta),\; \text{and}, \; u_{ij}\leq \theta\abs{x-y}^{1/2},
\\[5pt]
u_{ij} - \theta\abs{x-y}^{1/2}, & \text{otherwise}.
\end{array}
\right.
\]
It is easy to see that \eqref{46} is satisfied with the above choice of $\bar{u}$. Thus it remains to show that
$\bar{u}\in\M(x)$. It is enough to show that for $i\in\cA^c(\theta)$, we have
\begin{equation}\label{47}
\sum_{j:j\neq i} \bar{u}_{ij} x_j^-\;\leq\; x_i^+.
\end{equation}
Now for $i\in\cA^c(\theta)$,
\begin{align*}
\sum_{j:j\neq i} \bar{u}_{ij} x_j^- &\;\leq\; \sum_{u_{ij}>\theta\abs{x-y}^{1/2}}(u_{ij}-\theta\abs{x-y}^{1/2})(y_j^- + \abs{x-y})
\\[5pt]
&\;\leq\; \sum_{u_{ij}>\theta\abs{x-y}^{1/2}}u_{ij}y_j^- + \abs{x-y}\sum_{j: j\neq i} u_{ij}
-\theta\abs{x-y}^{1/2}\sum_{u_{ij}>\theta\abs{x-y}^{1/2}}y_j^-
\\[5pt]
&\;\leq\; y_i^+ + \abs{x-y}\sum_{j: j\neq i} u_{ij}
-\theta\abs{x-y}^{1/2}\sum_{u_{ij}>\theta\abs{x-y}^{1/2}}y_j^-
\\[5pt]
&\;\leq\; x_i^+ + d\, \abs{x-y}
-\theta\abs{x-y}.
\end{align*}
Thus choosing $\theta>d$ we see that \eqref{47} holds for $i\in\cA^c(\theta)$ and $\bar{u}\in\M(x)$.
\end{proof}

Now we are ready to prove Thereom~\ref{T-HJB}.
\begin{proof}[Proof of Theorem~\ref{T-HJB}]
Existence of solution $V_\eps\in\calC^2$ and optimality of $\varrho_\eps$ follows form Lemma~\ref{L4.6}. Uniqueness
of $(V_\eps, \varrho_\eps)$ can be obtained following the same arguments as in \cite[Theorem~3.7.12(iii)]{ari-bor-ghosh}.
Hence (i) and (ii) follows.
Now we argue that for any $r>0$, 
\begin{equation}\label{48}
\sup_{\eps\in(0, 1)}\sup_{\B_r}\abs{V_\eps}\; <\; \infty.
\end{equation}
We recall that $V_\eps$ is obtained as a limit of $\bar{V}^\al_\eps=V^\al_\eps-V^\al_\eps(0)$ where $V^\al_\eps$ is
given by \eqref{20}. From \cite[Lemma~3.5]{ari-bis-pang} (see also \cite[Lemma~3.6.3]{ari-bor-ghosh}) one obtains
that $\sup_{\eps\in(0, 1)}\sup_{x\in\B_r}\abs{\bar{V}^\al_\eps(x)}\;<\infty$. This gives us \eqref{48}. Therefore combining
\eqref{48}, \eqref{41} and the arguments in Lemma~\ref{L4.6} we find a constant $\kap$ such that
$$\sup_{\eps\in(0, 1)}\abs{V_\eps(x)}\; \leq \kap(1+\abs{x}^m),\quad \forall\; x\in\Rd.$$ 
Hence applying the theory of elliptic PDE's we see that the family $\{V_\eps\}$ bounded in $\Sobl^{2, p}(\Rd)$, for $p\geq d$. 
Since $\Sobl^{2, p}(\Rd), p>2d,$ is compactly embedded in $\calC^{1, \beta}_{\mathrm{loc}}(\Rd), \beta\in(0, 1/2)$,
we obtain $\{V_\eps\}$ bounded in $\calC^{1, \beta}_{\mathrm{loc}}(\Rd), \beta\in(0, 1/2)$. Thus we have
$V\in \Sobl^{2, p}(\Rd)\cap \calC^{1, \beta}_{\mathrm{loc}}(\Rd)$ with $p>2d$, and $\beta\in(0, 1/2)$, such that
$V_\eps\to V$ in $\Sobl^{2, p}(\Rd)\cap \calC^{1, \beta}_{\mathrm{loc}}(\Rd)$ along some sub-sequence of
$\eps\to 0$. Letting $\eps\to 0$ in \eqref{eps-HJB} we see that $(V, \varrho)$ satisfies \eqref{HJB}. Using 
Lemma~\ref{lem-holder-0} and regularity property of non-degenerate elliptic  operator we get $V\in\calC^2(\Rd)$.

(iv): Therefore there exists a classical solution $(V, \varrho_{*})$ to \eqref{HJB}. We now show that there exists a measurable minimizing selector
of \eqref{HJB}. We first show that the map $\chi:\Rd\to 2^{\M}$, defined as $\chi(x)=\M(x)$, is measurable. To check measurability we
need to show that for any closed $F\subset \M$,
$$\chi^{\ell}(F)\;\df\;\{x\in\Rd\; :\; \chi(x)\cap F\neq \emptyset\},$$
is a Borel set (see \cite[pp. 557]{ali-bor}). $\chi^\ell$ is referred to as the \textit{lower-inverse} of $\chi$. In fact, we show that $\chi^{\ell}(F)$
is a closed set whenever $F$ is closed. Let $x_n\to x$, as $n\to\infty$, for a sequence $\{x_n\}\subset\chi^\ell(F)$. Then there exists $u_n$ such that
$u_n\in\chi(x_n)\cap F$ and $u_n\in\M(x_n)$. Now $\M$ being compact there exists $u\in F$ satisfying $u_n\to u$ as $n\to\infty$. It is easy to
check that $u\in\M(x)$. Thus $\chi(x)\cap F\neq\emptyset$ implying $x\in\chi^\ell(F)$. This shows that $\chi^\ell(F)$ is a Borel set.
Thus $\chi$ is a measurable correspondence \cite[Chapter~18]{ali-bor}. Since $\chi$ is a compact set-valued map it is a weakly-measurable correspondence \cite[Theorem~18.10]{ali-bor}.
Hence from Filippov's implicit function theorem \cite[Theorem~18.17]{ali-bor} we obtain that there exists a measurable selector $u:\Rd\to\M$
such that $u(x)\in\M(x)$ for all $x$, and
\begin{equation}\label{100}
\frac{1}{2}\sum_{i, j=1}^da_{ij}\partial_{ij} V(x) + b(x, u(x))\cdot \grad V(x) + \tilde{r}(x, u(x)) = \varrho_{*}.
\end{equation}
On the other hand, we have $V\in\sorder(\abs{x}^k), k>m$. Therefore applying Dynkin's formula to \eqref{100} with a similar argument as in Lemma~\ref{L4.6} we obtain
that $u$ is optimal for $\varrho_{*}$.

(iii) and (v): We have already shown above that $V_\eps\to V$, along some subsequence as $\eps\to 0$, in $\Sobl^{2, p}(\Rd), \, p\geq 1$. Therefore
to get the convergence of full sequence it is enough to establish the uniqueness of the limit. Form Lemma~\ref{L4.5}(a) and \eqref{42} we observe that
$$V(x) \leq \liminf_{r\downarrow 0} \inf_{v\in\Tilde{\Uadm}_{\rm SM}} 
\Exp^v_x\Bigl[\int_0^{\tc_r}\bigl(\Tilde{r}(X(s), v(X(s))) -\varrho_{*}\bigr)\D{s}\Bigr].$$
Now uniqueness can be obtained following a similar argument as in \cite[Theorem~3.7.12(iii)]{ari-bor-ghosh}.
\end{proof}

\section{Asymptotic Optimality}\label{S-optimality}

The goal of this section is to prove Theorem~\ref{T-optimality}. In Theorem~\ref{T-lower} below we show that $\varrho_{*}$ is an
asymptotic lower bound for the value functions $\widehat{V}^n(\Hat{X}^n(0))$ as $n\to\infty$. Then using Theorem~\ref{T-HJB}
we construct a sequence of
 admissible policies for the queueing systems and show in Theorem~\ref{T-upper} that the admissible policies are
$\eps$-optimal for the value functions as $n\to\infty$.

Recall the diffusion scaled process $\Hat{X}^n, \Hat{Z}^n$ and $\Hat{Q}^n$ from \eqref{dc1} and
their relation \eqref{dc2}
\begin{align}\label{60}
\Hat{X}^{n}_{i}(t) &\;=\;\Hat{X}^{n}_{i}(0) +\ell_{i}^{n} t 
+ \mu_{i}^{n} \int_{0}^{t} (\Hat{X}_{i}^{n})^-(s)\,\D{s}-\sum_{j:j\neq i}\mu^n_{ij}\int_0^t\Hat{Z}^n_{ij}(s)\, \D{s}
- \gamma_{i}^{n} \int_{0}^{t} \Hat{Q}^{n}_{i}(s)\,\D{s}\\[5pt]
&\mspace{20mu} + \Hat{M}_{A,i}^{n}(t) - \Hat{M}_{S,i}^{n}(t)-\sum_{j:j\neq i} \Hat{M}_{S, ij}^n(t)
- \Hat{M}_{R,i}^{n}(t)\,,\nonumber
\end{align}
where $\ell^{n}= (\ell_{1}^{n},\dotsc,\ell_{d}^{n})\transp$ is defined as
\begin{equation*}
\ell_{i}^{n} \;\df\;\frac{1}{\sqrt{n}}(\lambda_{i}^{n}
- \mu_{i}^{n} n)= \frac{\lam^n_i-n\lam_i}{\sqrt{n}}-\sqrt{n}(\mu^n_i-\mu_i)\,,
\end{equation*}
and (see \eqref{dc3})
\begin{equation*}
\begin{split}
\Hat{M}_{A,i}^{n}(t) &\;\df\;\frac{1}{\sqrt{n}}(A_{i}^{n}(\lambda^{n}_{i} t )
- \lambda^{n}_{i} t)\,,\\[5pt]
\Hat{M}_{S,i}^{n}(t) &\;\df\;\frac{1}{\sqrt{n}} 
\left( S_{i}^{n}\left(\mu_{i}^{n} \int_{0}^{t} Z_{i}^{n}(s)\,\D{s}\right) 
- \mu_{i}^{n} \int_{0}^{t} Z_{i}^{n}(s)\,\D{s} \right)\,,\\[5pt]
\Hat{M}_{S,ij}^{n}(t) &\;\df\;\frac{1}{\sqrt{n}} 
\left( S_{ij}^{n}\left(\mu_{i}^{n} \int_{0}^{t} Z_{ij}^{n}(s)\,\D{s}\right) 
- \mu_{i}^{n} \int_{0}^{t} Z_{ij}^{n}(s)\,\D{s} \right)\,,\\[5pt]
\Hat{M}_{R,i}^{n}(t) &\;\df\;\frac{1}{\sqrt{n}} 
\left( R_{i}^{n} \left(\gamma_{i}^{n} \int_{0}^{t} Q_{i}^{n}(s) \,\D{s} \right)
- \gamma_{i}^{n} \int_{0}^{t} Q_{i}^{n}(s) \,\D{s} \right),
\end{split}
\end{equation*}
are square integrable martingales w.r.t. the filtration $\{\calF^{n}_{t}\}$ with quadratic variations,
\begin{equation*}
\begin{split}
\langle \Hat{M}_{A,i}^{n}\rangle(t) &\;\df\;\frac{1}{n}\lam^n_i t\,,\quad
\langle\Hat{M}_{S,i}^{n} \rangle (t) \;\df\;\frac{1}{n}\mu_{i}^{n} \int_{0}^{t} Z_{i}^{n}(s)\,,\\[5pt]
\langle{\Hat{M}_{S,ij}^{n}}\rangle (t) &\;\df\;\frac{1}{n}
\mu_{i}^{n} \int_{0}^{t} Z_{ij}^{n}(s)\,\D{s} \,,\\[5pt]
\langle \Hat{M}_{R,i}^{n} \rangle (t) &\;\df\;\frac{1}{n}
\gamma_{i}^{n} \int_{0}^{t} Q_{i}^{n}(s) \,\D{s}.
\end{split}
\end{equation*}
Also recall $\Hat{U}^n(t)$ from \eqref{02} where we have $\hat{Z}^n_{ij} = \hat{U}_{ij}\, (\hat{X}^n_{j})^-$. We also have
$\hat{U}(t)\in \M(\Hat{X}^n(t))$ a.s. for all $t\geq 0$. Define $b^n:\Rd\times\M\to\Rd$ as
$$b^n_i(x, u) =\ell^n_i + \mu_i^n x_i^- -\sum_{j:j\neq i}\mu^n_{ij} u_{ij}x_j^- -\gam_i^n q(x, u),$$
where 
$$q_i(x, u) = x_i^+- \sum_{j: j\neq i} u_{ij} x_j^-.$$
By $\Uadm^n$ we denote the set of all admissible controls. The following result establishes uniform stability
of the queueing systems.

\begin{lemma}\label{lem-queue-stability}
There exists $n_0\geq 1$, such that for any $k\geq 1$,
\begin{align}\label{68}
\sup_{n\geq n_0}\,\sup_{Z\in\Uadm^n}\,\limsup_{T\to\infty}\frac{1}{T}
\Exp\Bigl[\int_0^T\abs{\Hat{X}^n(t)}^k\,\D{t} \Bigr]\; <\; \infty.
\end{align}
\end{lemma}

\begin{proof}
Without loss of generality, we take $k\geq 2$. Let $\varphi(x) \; \df\; \abs{x}^k$. Recall that $\Delta f(t)$ denotes
jump of a function $f:\R_+\to\R$ at time $t$. Components of $\Hat{X}^n$ jumps due to the jumps of their martingale parts.
Since the optional quadratic variation between martingales that corresponds to different components is $0$
(see \cite[Lemma~9.1]{whitt-et-al})
no two component jumps at the same time.
Now applying It\^{o}'s formula on $\varphi$ (see \cite[Theorem~26.7]{kallenberg})
we obtain from \eqref{60} that
\begin{align}\label{61}
\Exp[\varphi(\Hat{X}^n(t))] &\; =\; \Exp[\varphi(\Hat{X}^n(0))] + \Exp\Bigl[\int_0^t b^n(\Hat{X}^n(s), \Hat{U}^n(s))\cdot
\grad\varphi(\Hat{X}^n(s))\, \D{s}\Bigr]\nonumber
\\[5pt]
&\hspace{.2in}\quad+ \frac{1}{2}\sum_{i=1}^d \Exp\Bigl[\int_0^t \Theta_i( \Hat{Z}^n(s), \Hat{Q}^n)
\cdot\partial_{ii}\varphi(\Hat{X}^n(s))\, \D{s}\Bigr] + \nonumber
\\[5pt]
&\;\;\Exp\sum_{s\leq t}\Bigl[\Delta\varphi(\Hat{X}(s))-\sum_{i=1}^d \partial_i\varphi(\Hat{X}(s-))\Delta \Hat{X}^n_i(s)
-\frac{1}{2}\langle \Delta \Hat{X}^n(s), D^2\varphi(\Hat{X}^n(s-))\, \Delta \Hat{X}^n(s)\rangle\Bigr],
\end{align}
where $D^2\varphi$ denotes Hessian of $\varphi$ and,
$$\Theta_i(z, y)\;\df\; \frac{1}{n}\lam_i^n + \frac{1}{\sqrt{n}}\mu_i^n z^n_i + \mu^n_i +
\frac{1}{\sqrt{n}}\sum_{j: j\neq i} z_{ij} + 
\frac{1}{\sqrt{n}}\gam^n_i y_i.$$ Using \eqref{HWpara} 
we can choose $n$ large enough so that $\min_{i}(\mu^N_i\wedge\gam^n_i)>0$. An argument similar to \eqref{17}-\eqref{18}
shows that for some positive constants $\kap_1, \kap_2 $, independent of $n$, we have
\begin{equation}\label{62}
b^n(\Hat{X}^n(t), \Hat{U}^n(t))\cdot \grad\varphi(\Hat{X}^n(t))\;\leq \; \kap_1 - \kap_2\,\abs{\Hat{X}^n(t)}^k, \quad a.s., 
\quad \text{for all}\;\; t\geq 0.
\end{equation}
Also from \eqref{xp3} we obtain $\Hat{Z}^n_i=(\Hat{X}^n_i)^-$ for all $i$. Thus for all $t\geq 0,$ we have
\begin{equation}\label{63}
\Big|\Theta_i( \Hat{Z}^n(t), \Hat{Q}^n(t))\cdot \partial_{ii}\varphi(\Hat{X}^n(t))\Big|\;
\leq \kap_3\, (1+ \abs{\Hat{X}^n(t)}^{k-1}), \quad \text{for all}\; \; i\in\{1, \ldots, d\},
\end{equation} 
for some constant $\kap_3$ and all large $n$, where we use \eqref{xp2}. We observe that 
$\abs{\Delta \hat{X}^n(t)}\leq \frac{1}{\sqrt{n}}$ for all $t$. Hence a straightforward calculation gives us
\begin{align}\label{64}
&\Bigl[\Delta\varphi(\Hat{X}(s))-\sum_{i=1}^d \partial_i\varphi(\Hat{X}(s-))\Delta \Hat{X}^n_i(s)
-\frac{1}{2}\langle \Delta \Hat{X}^n(s), D^2\varphi(\Hat{X}^n(s-))\, \Delta \Hat{X}^n(s)\rangle\Bigr]\nonumber
\\[5pt]
&\leq \frac{\kap_4}{\sqrt{n}}\, (1+ \abs{\Hat{X}^n(s-)}^{k-2})\sum_{i=1}^d (\Delta \Hat{X}^n_i(s))^2.
\end{align}
Since $\sum_{s\leq t}(\Delta \Hat{X}^n_i(s))^2\;\df\; [\Hat{X}^n_i](t)$ where $[\Hat{X}^n_i]$ is the optional quadratic variation,
we get from \eqref{64} that
\begin{align}\label{65}
&\Exp\sum_{s\leq t}\Bigl[\Delta\varphi(\Hat{X}(s))-\sum_{i=1}^d \partial_i\varphi(\Hat{X}(s-))\Delta \Hat{X}^n_i(s)
-\frac{1}{2}\langle \Delta \Hat{X}^n(s), D^2\varphi(\Hat{X}^n(s-))\, \Delta \Hat{X}^n(s)\rangle\Bigr]\nonumber
\\[5pt]
&\leq \frac{\kap_4}{\sqrt{n}}
\sum_{i=1}^d \Exp\Big[ \int_0^t \Theta_i(\Hat{Z}^n(s), \Hat{Q}^n)\Big(1+ \abs{\Hat{X}^n_i(s)}^{k-2}\Big) \D{s}\Big]\nonumber
\\[5pt]
&\leq \frac{\kap_5}{\sqrt{n}}\Bigl(t + \Exp\Big[ \int_0^t \abs{\Hat{X}^n(s)}^{k-1} \D{s}\Big]
\Bigr),
\end{align}
for some constant $\kap_5$, where we use the fact that $[\Hat{X}^n_i]-\langle \Hat{X}^n_i\rangle $ is also a martingale.
Therefore combining \eqref{61}--\eqref{63}, and \eqref{65} we obtain constants $\kap_6, \kap_7>0$, independent of $\Uadm^n$,
such that
$$\Exp[\varphi(\Hat{X}^n(t))] \; =\; \Exp[\varphi(\Hat{X}^n(0))] +\kap_6 t -\kap_7 \Exp\Big[\int_0^t\abs{\Hat{X}^n(s)} \D{s}\Big],$$
for all large $n$. Since $\{\Hat{X}^n(0)\}$ is bounded, we obtain \eqref{68} from the above display.
\end{proof}
For any $n\geq 1$ and $(\Hat{X}^n, \hat{U}^n)$ satisfying \eqref{dc2} and \eqref{02} we define mean-empirical measures 
$\xi^n_t\in\mathcal{P}(\Rd\times\M)$ as follows: For Borel $A\in\mathcal{B}(\Rd)$, $B\in\mathcal{\M}$,
\begin{equation}\label{70}
\xi^n_t(A\times B)\; \df\; \frac{1}{t}\Exp\Bigl[\int_0^t 1_{A\times B}(\Hat{X}^n(s), \Hat{U}^n(s))\, \D{s}\Bigr], \quad t\,>\, 0\, .
\end{equation}
Let 
$$\Tilde{\Theta}^n_i\; \df\; \frac{1}{n}\lam_i^n + \frac{1}{\sqrt{n}}\mu_i^n x_i^- + \mu^n_i +
\frac{1}{\sqrt{n}}\sum_{j: j\neq i} u_{ij}x_j^- + 
\frac{1}{\sqrt{n}}\gam^n_i q_i(x, u),$$
where $q$ is given by \eqref{0000}. From \eqref{HWpara} we have $\Tilde{\Theta}^n_i\to 2 \lam_i$, as $n\to\infty$, uniformly
on compacts. 


\begin{lemma}\label{L5.2}
Consider $n\geq n_0$, where $n_0$ is given by Lemma~\ref{lem-queue-stability}.
For $\Hat{U}^n\in\Uadm^n,$ we define $\xi^n_t$ as in \eqref{70}. Then the collection
$\{\xi^n, \, t\,>\,0\}$ is tight, as $t\to\infty$, and if $\uppi^n$ is a sub-sequential limit of $\{\xi^n, \, t\,>\,0\}$, then we have
$$\int_{\Rd\times\M} 1_{\M^c(x)}(u) \uppi^n(\D{x}, \D{u})\;=\;0.$$
\end{lemma}

\begin{proof}
From \eqref{68} we obtain that 
$$\limsup_{T\to\infty}\int_{\Rd\times\M}\abs{x}^k\; \xi^n_T(\D{x}, \D{u})\, <\infty,$$
for any $k\geq 1$. This implies that $\{\xi^n, \, t\,>\,0\}$ is tight. Let $\uppi^n$ be a sub-sequential limit and
$\xi^n_{t_l}\to \uppi^n$ as $t_l\to\infty$. Since $\Hat{U}^n(\Hat{X}^n(s))\in\M(\Hat{X}^n(s))$ for all $s$, we have
from \eqref{70} that
\begin{equation}\label{71}
\int_{\Rd\times \M} 1_{\M^c(x)}(u) \xi^n_{t}(\D{x}, \D{u})=0, \quad \forall\; t\; >\; 0.
\end{equation}
We note that $(x, u)\mapsto 1_{\M^c(x)}(u)$ is a lower-semicontinuous function. Thus there exists a sequence of 
bounded, continuous functions $g_j$ such that $g_j(x, u)\nearrow 1_{\M^c(x)}(u)$ as $j\to\infty$ \cite[Proposition~2.1.2]{krantz}.
Therefore \eqref{71} gives 
$$\int_{\Rd\times \M} g_j(x, u)\, \uppi^n(\D{x}, \D{u})=\lim_{t_l\to\infty} \int_{\Rd\times \M} g_j(x, u)\, \xi^n_{t_l}(\D{x}, \D{u})=0.$$
Now let $j\to\infty$, to complete the proof.
\end{proof}

Now we establish asymptotic lower-bound of the value functions $\widehat{V}^n$.

\begin{theorem}\label{T-lower}
As $n\to\infty$, $\widehat{V}^n(\Hat{X}^n(0))\;\geq \;\varrho_{*}$, where $\varrho_{*}$ is given by \eqref{0002}.
\end{theorem} 

\begin{proof}
Consider a sequence $Z^n\in\Uadm^n$ and let $\{\xi^n_t\}$ be the associated mean-empirical measures as defined 
in \eqref{70}. From \eqref{L5.2} we obtain that for $n\geq n_0$,  the collection $\{\xi^n_t\}$ is tight. Let 
$\uppi^n$ be a subsequential limit of $\{\xi^n_t\}$ as $t\to \infty$. Taking $k>m$ in Lemma~\ref{lem-queue-stability},
we obtain that
\begin{equation}\label{72}
\sup_{n\geq n_0} \int_{\Rd\times\M}\Tilde{r}(x, u) \, \uppi^n(\D{x}, \D{u})\; <\; \infty, \quad\text{and},\;\;
\sup_{n\geq n_0} \int_{\Rd\times\M}\abs{x}^k \, \uppi^n(\D{x}, \D{u})\; <\; \infty.
\end{equation}
Thus from \eqref{72} we see that the sequence $\{\uppi^n\; : \; n\geq 1\}$ is also a tight sequence.
Let $\uppi$ be a sub-sequential limit of $\{\uppi^n\; : \; n\geq 1\}$ as $n\to\infty$. We show that
$\uppi\in\mathcal{G}$ where $\mathcal{G}$ is given by \eqref{0001}.
Consider $f\in\calC^2_c(\Rd)$ and apply It\^{o}'s formula
in \eqref{60} to obtain
\begin{align*}
\Exp[f(\Hat{X}^n(t))] &\; =\; \Exp[f(\Hat{X}^n(0))] + \Exp\Bigl[\int_0^t b^n(\Hat{X}^n(s), \Hat{U}^n(s))\cdot
\grad f(\Hat{X}^n(s))\, \D{s}\Bigr]
\\[5pt]
&\hspace{.2in}\quad+ \frac{1}{2}\sum_{i=1}^d \Exp\Bigl[\int_0^t \Tilde\Theta^n_i(\Hat{X}^n(s), \Hat{U}^n(s))
\cdot\partial_{ii} f(\Hat{X}^n(s))\, \D{s}\Bigr] +
\\[5pt]
&\;\;\Exp\sum_{s\leq t}\Bigl[\Delta f(\Hat{X}(s))-\sum_{i=1}^d \partial_i f(\Hat{X}(s-))\Delta \Hat{X}^n_i(s)
-\frac{1}{2}\langle \Delta \Hat{X}^n(s), D^2 f(\Hat{X}^n(s-))\, \Delta \Hat{X}^n(s)\rangle\Bigr],
\end{align*}
and therefore dividing by $t$, we get
\begin{align*}
\frac{\Exp[f(\Hat{X}^n(t))]}{t} &\; =\; \frac{\Exp[f(\Hat{X}^n(0))]}{t} + 
\int_{\Rd\times\M} b^n(x, u)\cdot
\grad f(x)\, \xi^n_t(\D{x}, \D{u})
\\[5pt]
&\hspace{.2in}\quad+ \frac{1}{2}\sum_{i=1}^d \int_{\Rd\times\M} \Tilde\Theta^n_i(x, u)
\cdot\partial_{ii} f(x)\, \xi^n_t(\D{x}, \D{u})
\\[5pt]
&\;\;\frac{1}{t}\Exp\sum_{s\leq t}\Bigl[\Delta f(\Hat{X}(s))-\sum_{i=1}^d \partial_i f(\Hat{X}(s-))\Delta \Hat{X}^n_i(s)
-\frac{1}{2}\langle \Delta \Hat{X}^n(s), D^2 f(\Hat{X}^n(s-))\, \Delta \Hat{X}^n(s)\rangle\Bigr].
\end{align*}
Now let $t\to\infty$, and use a similar argument as in \eqref{65} to obtain
\begin{equation}\label{73}
\int_{\Rd\times\M} \Big(\frac{1}{2}\sum_{i=1}^d\Tilde\Theta^n_i(x, u)\cdot\partial_{ii} f(x) + b^n(x, u)\cdot\grad f(x)\Big)\,
 \uppi^n(\D{x}, \D(u)) = \order(\frac{1}{\sqrt{n}}).
\end{equation}
therefore letting $n\to\infty$, in \eqref{73} and using locally uniform convergence
property of $\Tilde{\Theta}^n, b^n$ we get
$$\int_{\Rd\times\M} L^u f(x)\, \uppi(\D{x}, \D(u))\; =\; 0,$$
where $L^u$ is given by \eqref{E3.3}. Therefore to show $\uppi\in\mathcal{G}$ it remains to prove that
$$\int_{\Rd\times\M}1_{\M^c(x)}(u)\, \uppi(\D{x}, \D{u})=0.$$
But this follows using the second part of Lemma~\ref{L5.2} and lower semicontinuity of the map. We also have 
$$\liminf_{n\to\infty} \int_{\Rd\times\M}\tilde{r}(x, u)\, \uppi^n(\D{x}, \D{u})\;\geq  \int_{\Rd\times\M}\tilde{r}(x, u)\, \uppi(\D{x}, \D{u}).$$
Hence using Theorem~\ref{T-exist}, we conclude that
$$\liminf_{n\to\infty}\, \limsup_{T\to\infty}\, \frac{1}{T} \Exp\Bigl[\int_0^T r(\Hat{Q}^n(s))\,\D{s}\Bigr]\; \geq \varrho_{*}.$$
\end{proof}
 Next we proceed to prove the asymptotic upper bound. The idea is to construct a sequence of admissible policies that
 achieves $\varrho_*$. One main obstacle with such construction is that the minimizer of the HJB \eqref{HJB} 
 might not be continuous,
 in general. Therefore we use the perturbed HJB \eqref{eps-HJB}. Let $u:\Rd\to\M$ be a continuous function and $u(x)\in\M(x)$
 for all $x$. Using $u$ we construct an admissible policy for every $n$ as follows. Recall that $\lfl a\rfl$ denotes the largest 
 integer small or equal to $a\in\R$. For $X^n(t)\in\Rd_+$, we define,
 \begin{align*}
 Z^n_i(t) & \;\df \; X^n(t)\wedge n, 
 \\[5pt]
 Z^{n}_{ki}(t) & \;\df\; \lfl u_{ki}(\Hat{X}^n(t)) (X^n_i(t)-n)^-\rfl, \quad i\neq k,
 \end{align*}
 where $\Hat{X}^n$ denotes the scaled version of $X^n$ under diffusion settings. We also define
 $$Q^n_i(t) = X^n_i(t) - Z^n_i(t) - \sum_{j: j\neq i} Z^n_{ij}(t), \quad\text{for}\;\; i\in\{1, \ldots, d\}.$$
We check that for $i\in\{1, \ldots, d\}$
\begin{align*}
Z^n_i(t) + \sum_{k: k\neq i} Z^n_{ki}(t) &\leq X^n(t)\wedge n + \sum_{k: k\neq i}u_{ki}(\Hat{X}^n(t)) (X^n_i(t)-n)^-
\\[5pt]
&\leq X^n(t)\wedge n + (X^n_i(t)-n)^-
\\[5pt]
&\leq n,
\end{align*}
and 
\begin{align*}
\sum_{j:j\neq i} Z^n_{ij}(t) &\leq \sum_{j:j\neq i}u_{ij}(\Hat{X}^n(t)) (X^n_j(t)-n)^-
\\[5pt]
&= \sqrt{n}\,\sum_{j:j\neq i}u_{ij}(\Hat{X}^n(t)) (\Hat{X}^n_j(t))^-
\\
&\leq (X^n_i(t)-n)^+.
\end{align*}
Therefore $Z^n\in\Uadm$ for all $n$. It is easy to see that
\begin{equation*}
\abs{\Hat{Z}^n_{ij}- u_{ij}(\Hat{X}^n_j)\,(\Hat{X}^n_j)^-}\; \leq \frac{1}{\sqrt{n}}, \quad \text{for all}\, \; i\neq j.
\end{equation*}

\begin{theorem}\label{T-upper}
We have
$$\limsup_{n\to\infty} \widehat{V}^n(\Hat{X}^n(0))\;\leq \;\varrho_{*},$$
where $\varrho_{*}$ is given by \eqref{0002}.
\end{theorem}

\begin{proof}
Since $\varrho_\eps\searrow\varrho_{*},$ as $\eps\to 0$, it is enough to show that there exists a sequence of admissible 
policy $Z^n\in\Uadm^n$ satisfying
\begin{equation}\label{75}
\limsup_{n\to\infty} J(\Hat{X}^n(0), \Hat{Z}^n)\;\leq \;\varrho_{\eps}\,,
\end{equation}
where $J(\Hat{X}^n(0), \Hat{Z}^n)$ is defined in \eqref{costd1}. Let $u_\eps$ be
the minimizing selector of \eqref{eps-HJB}. Now we construct a sequence of policy $Z^n$ as above given $u_\eps$. We define
empirical measures $\tilde\xi^n_t\in\mathcal{P}(\Rd)$ as
$$\tilde\xi^n_t(A) \;\df\; \frac{1}{t}\Exp_x\Big[ \int_0^t 1_A (\Hat{X}^n(s))\, \D{s}\Big], \quad t>0.$$
From Lemma~\ref{lem-queue-stability} we see that $\{\tilde\xi^n_t\}$ is tight as $t\to\infty$, and collection of the sub sequential limits
of  $\{\tilde\xi^n_t\}$, denoted by $\{\tilde\mu^n\}$, is also tight. Let $\tilde\mu$ be a sub-sequential limit of $\{\tilde\mu^n\}$. We claim that
\begin{equation}\label{76}
\int_{\Rd} \Big(\frac{1}{2}\sum_{i, j=1}^d a_{ij}\partial_{ii} f(x) + b(x, u_\eps(x))\cdot \grad f(x)\Big) \tilde\mu(\D{x})=0,
\quad \forall\; f\in\calC^2_c(\Rd).
\end{equation}
In fact, the claim follows from a similar argument as in Theorem~\ref{T-lower}. \eqref{76} shows that $\tilde\mu$ is the unique invariant
measure corresponds to the Markov control $u_\eps$. Thus to complete the proof we only need to show that
$$\lim_{n\to\infty} J(\Hat{X}^n(0), \hat{Z}^n)=\int_{\Rd}\tilde{r}(x, u_\eps(x))\, \tilde\mu(\D{x}).$$
In view of Lemma~\ref{lem-queue-stability}, to show \eqref{75} it is enough to show that for any $\psi\in\calC_c(\Rd)$, we have
\begin{equation}\label{80}
\lim_{n\to\infty}\int_{\Rd} \psi(x) \tilde{r}(x, u^n(x))\, \tilde{\mu}^n(\D{x}) =\int_{\Rd} \psi(x) \tilde{r}(x, u_\eps(x))\, \tilde{\mu}(\D{x}),
\end{equation}
where 
$$u^n_{ij}(x)\, x_j^- =\frac{1}{\sqrt{n}} \lfl \sqrt{n} u_\eps(x) x_j^-\rfl, \quad \text{for}\;\; i\neq j.$$
Since $\tilde{\mu}^n\to \tilde{\mu}$ and $|\tilde{r}(x, u^n(x))-\tilde{r}(x, u_\eps(x))|\to 0$, as $n\to\infty$,
uniformly on compacts, \eqref{80} follows.
\end{proof}

We conclude the article with two important  remarks.
\begin{remark}\label{con-rem-1}
It is evident from the arguments that one can replace $r(\Hat{Q}^n)$ by  $r(\Hat{Q}^n)+ h(\Hat{Z}^n)$ for some
convex function $h:\M\to\R_+$ that lies in $\calC_{\rm pol}(\Rd)$.
 Since $Z^n$ contains information about the idle times
and the customers in service, one may want to minimize the cumulative idle time by associating such cost. For instance,
if the service rate $\mu^n_{ij}$ (the rate at which class-$i$ customers are served at station $j$) is very small then it is logical to add
a cost of type $h(\Hat{Z}^n)$ with a high payoff  associated to $Z^n_{ij}$.
\end{remark}

\begin{remark}\label{con-rem-2}
One might also put additional constrains on the service mechanism
so that certain class of customers do not get served by some
pools. The arguments of this article still go through in that case. 
Note that under such constrains we have $i\nrightarrow j$ for some $i, j$ with $i\neq j$.
The only required change in such case is to restrict the corresponding
entries of the matrices in $\M$ to $0$ whenever $i\nrightarrow j$ for $i\neq j$.
\end{remark}

\subsection*{Acknowledgements} The author would like to thank the Editor and the anonymous referees for their helpful
comments in improving the article.
Author is grateful to Prof. Rami Atar for his comments on this project. 
Thanks are due to Prof. Kavita Ramanan for suggesting references related to this project. This research
is partly supported by an INSPIRE faculty fellowship.

\bibliographystyle{plain}
\bibliography{refs}

\end{document}